\documentclass[12pt]{amsart}
\usepackage{fullpage}
\usepackage[nofootnote,draft,silent]{fixme}
\usepackage{xypic}
\usepackage{enumerate}
\usepackage{amsthm}

\usepackage{color}

\usepackage{amsmath}
\usepackage{amssymb}
\usepackage{amsfonts}
\usepackage{hyperref}

\newtheorem{thm}{Theorem}[section]
\newtheorem*{thm*}{Theorem} 
\newtheorem{lem}[thm]{Lemma}

\newtheorem{prop}[thm]{Proposition}
\newtheorem{cor}[thm]{Corollary}
\theoremstyle{definition}
\newtheorem{Def}[thm]{Definition}
{\newtheorem{ex}[thm]{Example}}
{\newtheorem{rem}[thm]{Remark}}
{}
{\newtheorem{hyp}[thm]{Hypothesis}}

\renewcommand{\theenumi}{(\alph{enumi})}
\renewcommand{\labelenumi}{\alph{enumi})}  

\newcommand{\C}{\ensuremath{\mathbb{C}}}

\newcommand{\Q}{\ensuremath{\mathbb{Q}}}

\newcommand{\A}{\ensuremath{\mathbb{A}}}

\newcommand{\X}{\ensuremath{\mathcal{X}}}
\newcommand{\F}{\ensuremath{\mathcal{F}}}
\newcommand{\I}{\ensuremath{\mathcal{I}}}

\newcommand{\GL}{\operatorname{GL}}
\newcommand{\SL}{\operatorname{SL}}

\newcommand{\Ga}{\mathbb{G}_a}

\newcommand{\Galf}{\ensuremath{\mathrm{Gal}}}
\newcommand{\Hom}{\operatorname{Hom}}
\newcommand{\Aut}{\operatorname{Aut}}

\newcommand{\pr}{\operatorname{pr}}
\newcommand{\cha}{\operatorname{char}}

\newcommand{\Spec}{\operatorname{Spec}}
\newcommand{\Id}{\operatorname{id}}
\newcommand{\trd}{\operatorname{trdeg}}
\newcommand{\Frac}{\ensuremath{\mathrm{Frac}}}
\newcommand{\Ind}{\operatorname{Ind}}

\newcommand{\qq}{/\!/}
\newcommand{\ind}{\operatorname{Ind}^G_H(Y)}
\newcommand{\iind}{\operatorname{Ind}}
\newcommand{\indF}{\operatorname{Ind}^{G_F}_{H_F}}
\newcommand{\Fa}{{\bar{F}}}
\newcommand{\Ka}{{\bar{K}}}
\newcommand{\de}{\partial}
\newcommand{\del}{\ensuremath{\partial}}
\newcommand{\dt}{\widetilde{\partial}}
\newcommand{\id}{\operatorname{id}}

\renewcommand{\O}{\mathcal{O}}
\renewcommand{\H}{\mathcal{H}}
\newcommand{\M}{\mathcal{M}}
\renewcommand{\P}{\mathbb{P}}
\newcommand{\til}{\widetilde}

\title{Differential Embedding Problems over Complex Function Fields}
\author{Annette Bachmayr, David Harbater, Julia Hartmann, and Michael Wibmer}

\date{June 30, 2017}

\begin{document}

\thanks{The first author was funded by the Deutsche Forschungsgemeinschaft (DFG) -
grant MA6868/1-1. The second and third authors were supported in part by NSF FRG grant DMS-1463733. The third author acknowledges support from the Simons foundation (Simons fellowship).
The third and fourth author also acknowledge partial support by NSF grant DMS-1606334 (DART VII).
}

%-----------------------------------------------------------------------------

\begin{abstract}
			We introduce the notion of differential torsors, which allows the adaptation of constructions from algebraic geometry to differential Galois theory. Using these differential torsors, we set up a general framework for applying patching techniques in differential Galois theory over fields of characteristic zero. We show that patching holds over function fields over the complex numbers. As the main application, we prove the solvability of all differential embedding problems over complex function fields, thereby providing new insight on the structure of the absolute differential Galois group, i.e., the fundamental group of the underlying Tannakian category. 
	
\end{abstract}

\maketitle

%--------------------------------------------------------------------------------------------
\section*{Introduction}
%--------------------------------------------------------------------------------------------
This paper concerns embedding problems in differential Galois theory. Our main result generalizes two classical results over one-variable complex function fields:

\begin{itemize}
\item the solution of the inverse problem in differential Galois theory, and 
\item the solvability of all embedding problems in ordinary Galois theory. 
\end{itemize}

Inverse problems in Galois theory ask for the existence of Galois extensions of a given field with prescribed Galois group.  Embedding problems generalize this and are used to study how Galois extensions fit together in towers. 
In other words, the inverse problem asks which groups are epimorphic images of the absolute Galois groups, whereas solutions to embedding problems yield epimorphisms that in addition factor over a given epimorphism of groups. Therefore, solvability of embedding problems provides additional information about the structure of the absolute Galois group. This applies to both classical and differential Galois theory, where the absolute differential Galois group of a differential field $F$ is the fundamental group of the Tannakian category of all differential modules over $F$.

In the arithmetic context, the study of embedding problems has led to realizing all solvable groups as Galois groups over $\Q$ (\cite{Shaf}), and determining the structure of the maximal prosolvable extension of $\Q^{\mathrm{ab}}$ (\cite{Iwas}).  Solving embedding problems led to the proof of freeness of the absolute Galois group of a function field over an algebraically closed base field (\cite{Har95}, \cite{Pop95}), and contributed to the proof of Abhyankar's conjecture on fundamental groups in characteristic $p$ (see \cite{Serre90}, \cite{Raynaud94}, \cite{Har94}).  See \cite[Chapter~IX]{NSW} and \cite[Section~5]{Har03} for more about the arithmetic and geometric cases, respectively.

Differential Galois theory is an analog of Galois theory for (linear homogeneous) differential equations, over fields of characteristic zero. 
The symmetry groups that occur are no longer finite (or profinite), but rather are linear algebraic groups over the field of constants of the differential field. The corresponding inverse problem has been studied by a number of researchers. It was first solved for function fields of complex curves by Tretkoff and Tretkoff (\cite{Tretkoff}), as a relatively straightforward consequence of Plemelj's solution of the Riemann-Hilbert problem (\cite{Plemelj}). More generally, it has been solved when the base field is the field of functions of a curve over an algebraically closed field. See \cite{Hartmann} for the case of rational function fields; the general case, which appeared 
in \cite{Oberlies}, is based on the rational case and on Kovacic's trick (see also Proposition~\ref{Kovacic trick}). This solution built on prior work by Kovacic, Mitschi, and Singer (\cite{Kovacic}, \cite{Kovacic1}, \cite{singer}, \cite{mitschisinger1}, \cite{mitschisinger2}).  The differential inverse problem has also been solved for function fields over certain non-algebraically closed fields, including over the real numbers (\cite{Dyc}) and over fields of Laurent series (\cite{BHH}).

Embedding problems in differential Galois theory have been considered by several researchers (\cite{MatzatPut}, \cite{Hartmann}, \cite{Oberlies}, \cite{Ernst}). In fact, they were already used by Kovacic in his seminal work (\cite{Kovacic}, \cite{Kovacic1}) on the inverse problem and played a crucial role in the solution of the inverse problem over algebraically closed constant fields.  
In his thesis, Oberlies solved some types of differential embedding problems over function fields of curves over algebraically closed fields, including all differential embedding problems with connected cokernel (\cite{Oberlies}). However, the general solvability of embedding problems remained open, even in the classical case where the base field is $\C(x)$, although the inverse problem there had long been solved. In this paper, we close this gap and prove the following (see Theorem~\ref{thm ebp application}):

\begin{thm*} 
Every differential embedding problem over a one-variable complex function field (equipped with any nontrivial derivation), has a proper solution. 
\end{thm*}

To prove the theorem, we first introduce the notion of \textit{differential torsors}. Differential torsors generalize Picard-Vessiot rings (the differential analogues of finite Galois extensions).  We use a criterion given in
\cite{AmanoMasuokaTakeuchi:Hopf} to characterize those differential torsors
that are Picard-Vessiot rings. This can be viewed as a converse to a
well-known theorem of Kolchin. One advantage of working with differential torsors is that we can adapt constructions from algebraic geometry, such as passing to quotients, inducing differential torsors from subgroups, and transporting differential structures along morphisms.

The other main advantage of working with differential torsors is that this allows us to apply {\em patching}. We first deduce a patching result for differential torsors from the
corresponding assertion for torsors without differential structure in
\cite{HHKtorsor}. Building on this, we prove a patching result for Picard-Vessiot rings (Theorem \ref{thm patching PVR}) and another patching result designed for solving embedding problems (Theorem \ref{thm ebp new}). These results are stated in a very general framework amenable to further applications.

We then show that our patching results apply over finite extensions of $\C(x)$ by proving a factorization statement for matrices whose
entries are meromorphic functions on connected metric open subsets of a
compact Riemann surface (Lemma \ref{int-fact lem}). 
Similar factorization results were an
important ingredient in the solution of the Riemann-Hilbert problem. 

The strategy of using torsors for the purpose of Galois realizations has
previously been employed by other researchers, e.g.\ \cite{lourdes3},
\cite{lourdes2}, \cite{lourdes1}, usually to produce generic extensions for
specific groups (see also the references there). We expect that the finer notion of differential torsors may be a new tool in finding further generic
differential Galois extensions. It is already applied in an upcoming preprint on differential Galois theory over large fields of constants (\cite{BHHPpreprint}). Moreover, our patching results (Theorem \ref{thm patching PVR}, \ref{thm ebp new}) are used in a second upcoming preprint on that topic (\cite{BHHpreprint}). The explicit framework for applying patching to differential Galois theory over $\C(x)$ that we develop in Section~\ref{sec: Complex embedding problems} is used in another project (\cite{BWpreprint}).

\medskip

The organization of the paper is as follows.
In Section~\ref{sec: Differential torsors}, we define the notion of
differential torsors over a differential field of characteristic zero. 
We show that differential structures on a torsor
correspond to derivations that are invariant under translation. It is then
shown that a differential structure can be detected locally, by relating
invariant derivations to certain point derivations. Using this,
differential structures can be transported along morphisms of torsors,
under certain additional conditions. Finally, we show that simple
differential torsors with no additional constants correspond to
Picard-Vessiot rings. Section~\ref{sec: Patch Embed} extends the patching
result for torsors from \cite{HHKtorsor} to differential torsors, and
deduces a patching result for Picard-Vessiot rings. This is applied to give
a result on split differential embedding problems, under the hypothesis
that kernel and cokernel have suitable realizations as differential Galois
groups (Theorem~\ref{thm ebp new}).
In Section~\ref{sec: Complex embedding problems}, we describe a patching setup where the base field $F$ is the function field of a complex curve, and use Theorem~\ref{thm ebp new} to show that every split differential embedding problem over $F$ has a proper solution. Finally, using the results from Section~\ref{sec: Differential torsors} on transporting differential structures along morphisms, we show that the solution of arbitrary differential embedding problems can be reduced to the split case.
The appendix collects definitions and results
about group actions, torsors, quotients, and induction of torsors from a
subgroup, for affine group schemes of finite type over an arbitrary field.

\medskip

We would like to thank Phyllis Cassidy, Thomas Dreyfus, Ray Hoobler and Michael Singer for fruitful discussions. The authors also received helpful comments on an earlier version of the manuscript during the conference DART VII.

%----------------------------------------------------------------------------------------------------------------------------------------------------------------------------------------
\section{Differential torsors} \label{sec: Differential torsors}
%----------------------------------------------------------------------------------------------------------------------------------------------------------------------------------------
This section introduces the notion of differential torsors. We show that their differential structure is determined locally, and that it is possible to transport a differential structure along a morphism of torsors, under certain conditions. This will be used in Section~\ref{sec: Complex embedding problems} to reduce the solution of all differential embedding problems (over complex function fields) to that of all split differential embedding problems. We also prove that simple differential $G_F$-torsors correspond to Picard-Vessiot rings over $F$ with differential Galois group $G$. For notation and basic facts about torsors, see the the appendix to this manuscript.

Recall that a {\em differential ring} is a commutative ring $R$ equipped with a derivation $\de\colon R\to R$. A {\em differential homomorphism} 
is a ring homomorphism between differential rings
that commutes with the derivations. The ring of {\em constants} of $R$ is defined to be $C_R:=\{f\in R\,|\, \de(f)=0\}$.  If $R$ is a field then so is $C_R$.  We call $R$ {\em simple} if it has no non-trivial differential ideals (i.e., ideals closed under $\de$).

Throughout this section, $(F,\de)$ is a differential field of characteristic zero, and we let $K$ denote its field of constants $C_F$, which is algebraically closed in $F$.
The letter $G$ denotes an affine group scheme of finite type over $K$. Since $K$ has characteristic zero, $G$ is smooth and we will refer to $G$ as a {\em linear algebraic group}.

We will often consider the base change $G_F$ from $K$ to $F$ and we will view $F[G_F]=F\otimes_K K[G]$ as a differential ring extension of $F$ by considering all elements in $K[G]$ as constants.  We write $F[G]:=F[G_F]$.

%-----------------------------------------------------------------------------------------------------
\subsection{Differential torsors and invariant derivations}
%-----------------------------------------------------------------------------------------------------

\begin{Def}
A {\em differential $G_F$-space} is an affine $G_F$-space $X$ equipped with an extension of the given derivation $\de$ from $F$ to $F[X]$, such that the co-action $\rho \colon F[X] \to F[X] \otimes_F F[G]$ corresponding to the action $\alpha$ of $G_F$ on $X$ is a differential homomorphism (a definition of the terms \textit{$G_F$-space} and \textit{co-action} is given in \ref{subsec Gspaces}). 
We also call such an extension a {\em differential structure} on $X$.
A {\em morphism of differential $G_F$-spaces} $\phi \colon X\to Y$ is a morphism of affine varieties that is $G_F$-equivariant and such that the corresponding homomorphism $F[Y]\to F[X]$ is a differential homomorphism. 
A differential $G_F$-space $X$ is {\em simple} if
$F[X]$ is a simple differential ring.  Note that a $G_F$-space $X$ is a {\em differential $G_F$-torsor}
(i.e., a differential $G_F$-space that is a $G_F$-torsor) 
if and only if the 
left $F[X]$-linear extension $F[X]\otimes_ F F[X]\to F[X]\otimes_F F[G]$ 
of $\rho$ is a differential isomorphism, or equivalently $(\alpha,\pr_1):X \times G  \to X \times X$ is a differential isomorphism.
\end{Def}

The next lemma gives a criterion for when a derivation on a $G_F$-torsor defines a differential torsor.

Let $\Fa$ denote an algebraic closure of $F$. Note that $\del\colon F \to F$ uniquely extends to a derivation $\del\colon \Fa\to\Fa$. Let $\Ka$ denote the algebraic closure of $K=C_F$ in $\Fa$ and note that $C_\Fa=\Ka$. In particular, if $K$ is algebraically closed, $C_{\Fa}=K$. Let $\Gamma$ denote the Galois group of $\Fa$ over $F$.

\begin{lem} \label{lem: invariant derivation}
	Let $X$ be a $G_F$-torsor, and let $\de\colon F[X]\to F[X]$ be a derivation extending $\de\colon F\to F$. Then $X$ is a differential $G_F$-torsor if and only if $g\circ\del=\del\circ g$ on $F[X]\otimes_K \Ka$ for all $g\in G(\Ka)$. Here $g\in G(\Ka)$ is viewed as the automorphism $g\colon F[X]\otimes_K \Ka\to F[X]\otimes_K \Ka$ corresponding to $X\rightarrow X$, $x\mapsto x.g$, and $\de\colon F[X]\otimes_K \Ka\to F[X]\otimes_K \Ka$ is the unique extension of $\de$.
\end{lem}

\begin{proof}

		Assume that $X$ is a differential $G_F$-torsor. Let $g$ be an arbitrary element of $G(\Ka)$.  Then the left square in
	$$
	\xymatrix{
		F[X] \ar^-\rho[r] \ar_\de[d] & F[X]\otimes_K K[G] \ar^{\de\otimes \operatorname{id}_{K[G]}}[d] \ar^-{\id_{F[X]}\otimes g}[rr] & & F[X]\otimes_K\Ka \ar^{\de\otimes\id_{\Ka}}[d] \\
		F[X] \ar^-{\rho}[r] & F[X]\otimes_K K[G] \ar^-{\id_{F[X]}\otimes g}[rr] & & F[X]\otimes_K \Ka	
	}
	$$
	commutes. Since the right square commutes, the outer rectangle also commutes; i.e., $g\circ\del=\del\circ g$ for all $g \in G(\Ka)$.
	
Conversely, if $g\circ\del=\del\circ g$ for all $g \in G(\Ka)$, then the outer rectangle commutes. So $\rho(\de(f))$ and $\de(\rho(f))\in F[X]\otimes_K K[G]$ have the same image in $F[X]\otimes_K \Ka$ for all $g\in G(\Ka)$ and $f\in F[X]$. Since $G$ is reduced this implies that $\rho(\de(f))=\de(\rho(f))$; i.e., $X$ is a differential torsor.	
\end{proof}

If $K$ is algebraically closed, $G(\Ka)=G(K)\subset G(F)$, and Lemma \ref{lem: invariant derivation} simplifies to

\begin{cor} \label{cor: invariant derivation}
Assume $K=\Ka$, let $X$ be a $G_F$-torsor and let $\de\colon F[X]\to F[X]$ be a derivation extending $\de\colon F\to F$. Then $X$ is a differential $G_F$-torsor if and only if $g\circ\del=\del\circ g$ on $F[X]$ for all $g\in G(K)$.
\end{cor} 

Thus a derivation $\de: F[X]\to F[X]$ that extends $\de: F\to F$ is \emph{$G$-invariant} (i.e., 
$g\circ\del=\del\circ g$ for all $g\in G(\Ka)$) if and only if it turns the $G_F$-torsor $X$ into a differential $G_F$-torsor.

%-----------------------------------------------------------------------------------------------------
\subsection{Point derivations}
%-----------------------------------------------------------------------------------------------------

Our next goal is to show that a differential structure on a torsor can be detected locally, i.e., at a point.
For a linear algebraic group $G$, the space of $G$-invariant derivations on the coordinate ring of $G$ is isomorphic to the tangent space at the identity (\cite[Theorem 9.1]{Humphreys:Linear algebraic groups}). A similar construction applies in our setting, as we now discuss. 

 Let $R$ be a ring and let $M$ be an $R$-module. A {\em derivation} $\de:R\to M$ is an additive map that satisfies the Leibniz rule $\de(r_1r_2)=r_1\de(r_2)+r_2\de(r_1)$. Notice that it is necessary to specify an $R$-module structure on $M$ for the Leibniz rule to make sense.

Let $L$ be any differential field (e.g., $F$ or $\bar F$).  
Given an affine variety $X$ over $L$, together with an $L$-algebra $A$ and an $A$-point $x\in X(A)$, the map $L[X] \to A$ given by $f\mapsto f(x)$ defines an $L[X]$-module structure on $A$.  A {\em point derivation $\dt:L[X]\rightarrow A$ at $x$} is a derivation (with respect to this $L[X]$-module structure) that extends $\de:L\to L$.

Let $G$ be a linear algebraic group over $K$ and let $X$ be a $G_F$-torsor.  Consider an element $x\in X(\Fa)$ and a point derivation $\dt\colon F[X]\to\Fa$ at $x$. If $R$ is an $\Fa$-algebra and $y\in X(R)$, we may define a point derivation $\dt^y\colon F[X]\to R$ at $y$ as follows.
Let $g\in G(R)=G_\Fa(R)=\Hom_\Fa(\Fa\otimes_K K[G],R)$ be the unique element such that $y=x.g$. Now define $\dt^y$ as the composition 
$$\dt^y\colon F[X]\xrightarrow{\rho}F[X]\otimes_F F[G]=F[X]\otimes_K K[G]\xrightarrow{\dt\otimes \operatorname{id}_{K[G]}}\Fa\otimes_K K[G]\xrightarrow{g} R.$$
Note that $\dt\otimes \operatorname{id}_{K[G]}$ is a derivation, where $\Fa\otimes_K K[G]$ is considered as a $F[X]\otimes_K K[G]$-module via $F[X]\otimes_K K[G]\xrightarrow{x\otimes\id_{K[G]}}\Fa\otimes_K K[G]$. Therefore
$\dt^y$ is a derivation with respect to the $F[X]$-module structure on $R$ given by $g\circ(x\otimes \id_{K[G]})\circ\rho=x.g=y\in G(R)=\Hom_F(F[X],R)$.

For $\tau\in\Gamma:=\operatorname{Gal}(\Fa/F)$, the map $\tau(\dt):=\tau\circ\dt\colon F[X]\to\Fa$ is a point derivation at $\tau(x)$. We call $\dt$ {\em Galois-equivariant} if $(\tau(\dt))^x=\dt$ for all $\tau \in \Gamma$.

The next lemma lists some basic properties of point derivations.

\begin{lem}\label{point-derivation-properties}
With notation as above, suppose $\dt: F[X]\to \Fa$ is a point derivation at $x\in X(\Fa)$. 
Then the following hold:
\renewcommand{\theenumi}{(\alph{enumi})}
\renewcommand{\labelenumi}{(\alph{enumi})}
\begin{enumerate}
		\item \label{point lem a} $\dt^x=\dt$. 
		\item \label{point lem b} If $y,z\in X(\Fa)$ then $(\dt^y)^z=\dt^z$.
		\item \label{point lem c} If $\dt$ is Galois-equivariant and $y\in X(\Fa)$, then $\tau(\dt)^y=\dt^y$ for all $\tau \in \Gamma$.
		\item \label{point lem d} For $y\in X(\Fa)$ and $\tau\in \Gamma$ we have $\tau(\dt)^{\tau(y)}=\tau(\dt^y)$. In particular, $\tau(\dt^{\tau^{-1}(y)})=\tau(\dt)^{\tau(\tau^{-1}(y))}=\tau(\dt)^y$. 
	\end{enumerate}
\end{lem}

\begin{proof}
To prove \ref{point lem a} note that $x=x.1$. The assertion now follows from the fact that the two inner diagrams in
	$$
	\xymatrix{
		F[X] \ar^-\rho[r] \ar_{\id_{F[X]}}[rd] & F[X]\otimes_K K[G] \ar^-{\dt\otimes \operatorname{id}_{K[G]}}[rr] \ar^{\id_{F[X]}\cdot 1}[d] & & \Fa\otimes_K K[G] \ar^{\id_{\Fa}\cdot 1}[d] \\
		& F[X]\ar^\dt[rr] & & \Fa
	}
	$$
	commute, where $1$ denotes evaluation at $1 \in G(K)$.

	To prove \ref{point lem b}, first note that the two inner diagrams in
	$$
	\xymatrix{
		F[X]  \ar^\rho[rr] \ar_-\rho[d] & & F[X]\otimes_K K[G] \ar^{\dt\otimes \id_{K[G]}}[rr] \ar_{\id_{F[X]}\otimes\Delta}[d] & & \Fa\otimes_K K[G] \ar^{\id_{\Fa}\otimes\Delta}[d] \\
		F[X]\otimes_K K[G] \ar^-{\rho\otimes\id_{K[G]}}[rr] & & F[X]\otimes_K K[G] \otimes_K K[G] \ar^-{\dt\otimes \id_{K[G]\otimes K[G]}}[rr] & & \Fa\otimes_K K[G]\otimes_K K[G]	
	}
	$$
	commute. Therefore the outer rectangle also commutes. Let $g\in G(\Fa)$ such that $y=x.g$ and let $h\in G(\Fa)$ such that $z=y.h$. Then $z=x.gh$; and we see that $\dt^z$ is the upper right path 
from $F[X]$ to $\Fa\otimes_K K[G]\otimes_K K[G]$	
composed with $g\circ (h\otimes \operatorname{id})\colon \Fa\otimes_K K[G]\otimes_K K[G]	\to \Fa$, whereas $(\dt_y)^z$ is the lower left path composed with $g\circ (h\otimes \operatorname{id})$. This proves \ref{point lem b}.
	
	To prove \ref{point lem c}, note that by Galois-equivariance and \ref{point lem a} we have
	$\tau(\dt)^x=\dt^x$. Therefore $\tau(\dt)^y=(\tau(\dt)^x)^y=(\dt^x)^y=\dt^y$ by \ref{point lem b}.
	
	Finally for \ref{point lem d}, note that if $y=x.g$, then $\tau(y)=\tau(x).\tau(g)$. 
Since the diagram
	$$
	\xymatrix{
		F[X]\otimes_K K[G] \ar^-{\dt\otimes \id_{K[G]}}[rr] \ar_{\tau(\dt)\otimes\operatorname{id}_{K[G]}}[d]
		& & \Fa\otimes_K K[G] \ar^g[d] \\
		\Fa\otimes_K K[G] \ar^{\tau(g)}[rrd] & & \Fa \ar^\tau[d] \\
		& & \Fa
	}
	$$
	commutes, the claim follows.
\end{proof}

If $X$ is a $G_F$-torsor as above and $\de:F[X]\rightarrow F[X]$ is a derivation that extends the given derivation $\de:F \to F$, then there is an {\em induced} point derivation $\de_x$ at $x \in X(\Fa)$ obtained by evaluation at $x$; i.e., $\de_{x}\colon F[X]\xrightarrow{\de} F[X]\xrightarrow{x}\Fa$. 

The following explains the connection between differential structures on a torsor and point derivations.

\begin{prop} \label{prop: differential structure at point}
Let $x\in X(\Fa)$ be a fixed element. The assignment $\de\mapsto \de_x$ induces a bijection between the derivations $\del\colon F[X]\to F[X]$ that endow $X$ with the structure of a differential $G_F$-torsor and the point derivations $\widetilde{\del}\colon F[X]\to\Fa$ at $x$ that are Galois-equivariant.  In particular, $(\de_x)^y=\de_y$ for all $y \in Y(\Fa)$.
\end{prop}  

\begin{proof}
	Given a derivation $\de$ on $F[X]$ which endows $X$ with the structure of a differential $G_F$-torsor, let $\dt=\de_x$. This is a point derivation at $x$ by definition. To see that $\dt$ is Galois-equivariant, let $\tau\in\Gamma$ and let $g\in G(\Fa)$ be such that $x=\tau(x).g$. The square in
	$$
	\xymatrix{
		F[X] \ar^-\rho[r] \ar_\de[d] & F[X]\otimes_K K[G] \ar^{\de\otimes \id_{K[G]}}[d]  & \\
		F[X] \ar^-\rho[r] & F[X]\otimes_K K[G] \ar^-{\tau(x)}[r] & \Fa\otimes_K K[G] \ar^-g[r] & \Fa		
	}
	$$
	commutes. The upper path all the way to $\Fa$ is $\tau(\dt)^x$ and the lower path is $\dt$ because $x=\tau(x).g$. Thus $\dt$ is Galois-equivariant.
	
Conversely, given a point derivation $\dt$ that is Galois-equivariant, we 
want to associate to $\dt$ a derivation on $F[X]$ that extends $\de$ on $F$.
Let $j:\Fa[X]\to\Fa[G]$ correspond to the scheme isomorphism $G_\Fa \to X_\Fa$ defined by $g \mapsto x.g$;
and let $\id \in \Hom_{\Fa}(\Fa[X],\Fa[X]) = X_{\Fa}(\Fa[X])$ be the identity map on $\bar F[X]$, 
corresponding to an $\bar F[X]$-point $y_0$ on $X$.  Consider the point derivation
$\dt^{y_0}:F[X] \to \Fa[X]$.  If $f \in F[X]$, then $\dt^{y_0}(f) \in \bar F[X]$ defines a morphism 
$\de(f):X_{\Fa}\to\mathbb{A}_\Fa^1$.  
There is a unique $g_0 \in G(\Fa[X])$ satisfying $x.g_0 = y_0$; i.e., such that 
the composition 
$\Fa[X] \xrightarrow{j} \Fa[G] \xrightarrow{g_0} \Fa[X]$ is equal to $\id:\Fa[X] \to \Fa[X]$.
Thus $j=g_0^{-1}$.

We claim that if $R$ is any $\Fa$-algebra and if $y\in X_\Fa(R)$, then
$\de(f)(y)=\dt^y(f)$.  By the definition of $\dt^y$, it suffices to show that the composition
$\Fa[G] \xrightarrow{g_0} \Fa[X] \xrightarrow{y} R$ is equal to $g:\Fa[G] \to R$, where $g \in G(R)$ is the element satisfying $x.g = y$.  This last equality says that the composition 
$\Fa[X] \xrightarrow{j} \Fa[G] \xrightarrow{g} R$ is equal to $y:\Fa[X] \to R$.  Hence $g = y \circ g_0:\Fa[G] \to R$, proving the claim.

Using the Galois-equivariance of $\dt$, we will check that $\de(f)$ is in fact in $F[X]$. For this, it suffices to show that $\de(f)\in F[X]\otimes_F\Fa$ is fixed by the $\Gamma$-action. For $y\in X(\Fa)$ and $\tau\in\Gamma$ we find
	$$\tau(\de(f))(y)=\tau(\de(f)(\tau^{-1}(y)))=\tau(\dt^{\tau^{-1}(y)}(f))=\tau(\dt^{\tau^{-1}(y)})(f)=\dt^y(f)=\de(f)(y)$$
	using Lemma~\ref{point-derivation-properties}\ref{point lem c} and~\ref{point lem d}. Thus $\de(f)\in F[X]$ and we have constructed a well-defined map
	$$\de\colon F[X]\to F[X],$$
as desired.
It is now straightforward to check that $\de$ is a derivation and extends $\de:F\rightarrow F$.
	
	Our next goal is to show that $\de\colon F[X]\to F[X]$ endows $X$ with the structure of a differential $G$-torsor using Lemma \ref{lem: invariant derivation}. For $g\in G(\Ka)$ and $y\in X(\Fa)$ the square in
	$$
	\xymatrix{
		F[X] \ar^-\rho[r] & F[X]\otimes_K K[G] \ar^-{\id_{F[X]}\otimes g}[rr] \ar_{\dt_y\otimes \id_{K[G]}}[d] & & F[X]\otimes_K \Ka \ar^{\dt_y\otimes \id_{\Ka}}[d] & \\
		& \Fa\otimes_K K[G] \ar^-{\id_{\Fa}\otimes g}[rr] & & \Fa\otimes_K \Ka \ar[r] & \Fa
	}
	$$
	commutes. Tracking an element $f\in F[X]$ from the upper left to the lower right along both paths, we find that
	$$\dt^{y.g}(f)=\dt^y(g(f)),$$
	where we extend $\dt^y\colon F[X]\to \Fa$ to $\dt^y\colon F[X]\otimes_K \Ka\to \Fa$  by $\Ka$-linearity. Now
	$$g(\de(f))(y)=\de(f)(y.g)=\dt^{y.g}(f)=\dt^y(g(f))=\de(g(f))(y).$$
	Hence by Lemma \ref{lem: invariant derivation}, $X$ is a differential $G_F$-torsor with respect to $\de$.

	It remains to see that the constructed maps are inverse to each other.
	If we start with $\dt\colon F[X]\to \Fa$ and construct $\de\colon F[X]\to F[X]$ as above, then $\de(f)(x)=\dt^x(f)=\dt(f)$ for $f\in F[X]$ by Lemma~\ref{point-derivation-properties}\ref{point lem a}. So $\de_x$ is the given $\til \de$.
	
	Conversely, if we start with $\de\colon F[X]\to F[X]$ and define $\dt:=\de_x$, it remains to be seen that $\dt^y(f)=\de(f)(y)$ for all $f\in F[X]$ and $y\in X(\Fa)$. 
That is, we want to check that $(\de_x)^y=\de_y$.  Let $g\in G(\Fa)$ be such that $y=x.g$. Since the two inner diagrams in	
	$$
	\xymatrix{
		F[X] \ar^-\rho[r] \ar_\de[d] & F[X]\otimes_K K[G] \ar^{\de\otimes \id_{K[G]}}[d]   \\
		F[X] \ar^-\rho[r] \ar_-y[ddr] & F[X]\otimes_K K[G] \ar^{x\otimes\id_{K[G]}}[d] \\
		& \Fa \otimes_K K[G] \ar^{g}[d] \\
		& \Fa
	}
	$$
	commute, the outer diagram also commutes. This shows that $(\de_x)^y=\de_y$.
\end{proof}

%-----------------------------------------------------------------------------------------------------
\subsection{Transport of differential structures}
%-----------------------------------------------------------------------------------------------------

In this subsection, we study how differential structures behave under morphisms of torsors.

\begin{Def} \label{def: differential morphism}
	Let $\phi\colon G\to G'$ be a morphism of linear algebraic groups over $K$,  let $X$ be a differential $G_F$-torsor, and let $X'$ be a differential $G'_F$-torsor. A $G_F$-equivariant morphism $\psi\colon X\to X'$ is called \emph{differential} if the corresponding dual morphism $\psi^*\colon F[X']\to F[X]$ is differential (i.e., commutes with the derivation).
\end{Def}

The following lemma gives a local criterion for a morphism to be differential.

\begin{lem} \label{lemma: differential iff differential at point} Consider the situation of Definition \ref{def: differential morphism}, and fix $x \in X(\Fa)$. The 
morphism $\psi$ is differential if and only if  
	\begin{equation} \label{eqn: diagram for delta}
	\xymatrix{
		F[X'] \ar^{\psi*}[rr] \ar_-{\de_{\psi(x)}}[rd] & & F[X]  \ar^-{\de_x}[ld] \\
		& \Fa &	
	}	
	\end{equation}
	commutes (for this fixed $x$).
\end{lem}

\begin{proof}
Unraveling the definitions shows that $\psi$ is differential if and only if the diagram~(\ref{eqn: diagram for delta}) commutes {\em for all} $x\in X(\Fa)$. It remains to show that the commutativity for one fixed $x$ is sufficient. For $\tilde{x}\in X(\Fa)$, there exists $g\in G(\Fa)$ such that $\tilde{x}=x.g$; and then $\psi(\tilde{x})=x'.\phi(g)$, where $x'=\psi(x)$.
	Since the three inner diagrams in	
	$$
	\xymatrix{
		F[X'] \ar^{\psi^*}[rr] \ar[d] & & F[X] \ar[d] \\
		F[X']\otimes_K K[G'] \ar^{\psi^*\otimes\phi^*}[rr] \ar_{\de_{x'}\otimes \id_{K[G']}}[d] & & F[X]\otimes_K K[G] \ar^{\de_x\otimes \id_{K[G]}}[d] \\
		\Fa\otimes_K K[G'] \ar^{\id_{\Fa}\otimes\phi^*}[rr] \ar_{\phi(g)}[rd] & & \Fa\otimes_K K[G] \ar^g[ld] \\
		& \Fa &
	}
	$$
	commute, the outer diagram also commutes. That is,
	$$
	\xymatrix{
		F[X'] \ar^{\psi^*}[rr]  \ar_-{(\de_{x'})^{\psi(\tilde{x})}}[rd] & & F[X] \ar^{(\de_x)^{\tilde{x}}}[ld] \\
		& \Fa &
	}
	$$
commutes. By Proposition~\ref{prop: differential structure at point}, 
$(\de_{x'})^{\psi(\tilde{x})}=\de_{\psi(\tilde x)}$ and $(\de_x)^{\tilde x}=\de_{\tilde{x}}$, hence both paths from the upper left to the lower right in 
	$$
	\xymatrix{
		F[X'] \ar^{\psi^*}[r] \ar_{\de}[d] & F[X] \ar^{\de}[d]\\
		F[X'] \ar^{\psi^*}[r] & F[X] \ar^{\tilde{x}}[d] \\
		& \Fa	
	}
	$$
	yield the same result; i.e., $\psi^*(\de(f))(\tilde{x})=\de(\psi^*(f))(\tilde{x})$ for all $f\in F[X']$ and $\tilde{x}\in X(\Fa)$. Thus $\psi$ is differential.
\end{proof}

\begin{prop} \label{prop: moving differential structures}
	Let $\phi\colon G\to G'$ be a morphism of linear algebraic groups over $K$ and let $\psi\colon X\to X'$ be a $G_F$-equivariant morphism from a $G_F$-torsor $X$ to a $G'_F$-torsor $X'$ (where $G$ acts on $X'$ via $\phi$). Then:
\renewcommand{\theenumi}{(\alph{enumi})}
\renewcommand{\labelenumi}{(\alph{enumi})}
	\begin{enumerate}
		\item \label{moving a}
		If $X$ is differential, then there exists a unique differential structure on $X'$ such that $\psi$ is differential.
		\item \label{moving b} 
		If $X'$ is differential, $K$ is algebraically closed, and $\phi$ is surjective, then there exists a differential structure on $X$ such that $\psi$ is differential. 
	\end{enumerate}		
\end{prop}

\begin{proof}
	For part \ref{moving a}, let $x\in X(\Fa)$ and set $x'=\psi(x)\in X'(\Fa)$. Define $\de_{x'}:=\de_x\circ \psi^*:F[X']\rightarrow \Fa$.
	 We next show that $\de_{x'}$ is Galois-equivariant. If $\tau\in\Gamma$ and $g\in G(\Fa)$ such that $x=\tau(x).g$, then $x'=\tau(x').\phi(g)$.
	Since the three inner diagrams in	
		$$
		\xymatrix{
			F[X'] \ar^{\psi^*}[rr] \ar[d] & & F[X] \ar[d] \\
			F[X']\otimes_K K[G'] \ar^{\psi^*\otimes\phi^*}[rr] \ar_{\tau(\de_{x'})\otimes \id_{K[G']}}[d] & & F[X]\otimes_K K[G] \ar^{\tau(\de_x)\otimes \id_{K[G]}}[d] \\
			\Fa\otimes_K K[G'] \ar^{\id_{\Fa}\otimes\phi^*}[rr] \ar_{\phi(g)}[rd] & & \Fa\otimes_K K[G] \ar^g[ld] \\
			& \Fa &
		}
		$$
		commute, the outer diagram also commutes. So $\tau(\de_{x'})^{x'}=\tau(\de_x)^x\circ\psi^*=\de_x\circ\psi^*=\de_{x'}$, using that $\de_x$ is Galois-equivariant by Proposition~\ref{prop: differential structure at point}. Thus $\de_{x'}$ is Galois-equivariant, as asserted.  Hence by Proposition~\ref{prop: differential structure at point}, $\de_{x'}$ defines a differential structure on $X'$. By Lemma~\ref{lemma: differential iff differential at point} the morphism $\psi$ is differential. The uniqueness is clear from Proposition \ref{prop: differential structure at point} and Lemma~\ref{lemma: differential iff differential at point}.	
	
	\medskip
	
	To prove \ref{moving b}, we first assume that $F=\Fa$ is algebraically closed. Let $x\in X(F)$. Since $F$ is algebraically closed, every point derivation at $x$ is Galois-equivariant.  Hence Proposition~\ref{prop: differential structure at point} implies that the differential structures on $X$ are in bijection with the set $\operatorname{Der}_{\de,x}(F[X],F)$ of point derivations at $x$. Let $M_x\subseteq F[X]$ be the maximal ideal corresponding to $x$ and let $\mathcal{O}_{X,x}=F[X]_{M_x}$, the local ring at $x$. Recall that the (Zariski) tangent space $T_x X$ at $x$ is the dual $F$-vector space of $\mathfrak{m}_x/\mathfrak{m}_x^2$ where $\mathfrak{m}_x\subseteq \mathcal{O}_{X,x}$ is the maximal ideal. 
We then have a chain of bijections:
		\begin{equation} \label{eqn: bijections for tangent space}
	\operatorname{Der}_{\de,x}(F[X],F)\cong\operatorname{Der}_{\de,x}(\mathcal{O}_{X,x},F)\cong T_x X.
	\end{equation}
(Note that while $T_xX$ is an $F$-vector space, $\operatorname{Der}_{\de,x}(F[X],F)$ is not closed under addition of derivations, since the sum of two derivations extending $\de\colon F\to F$ extends $2\de$ rather than $\de$.)
Although $\dt\in\operatorname{Der}_{\de,x}(\mathcal{O}_{X,x},F)$ need not be $F$-linear, the induced map $\mathfrak{m}_x/\mathfrak{m}_x^2\to F$ given by $f\!\mod \mathfrak{m}_x^2  \mapsto\dt(f)$ is well defined and $F$-linear by the Leibniz rule; and this gives the forward direction of the second bijection
in (\ref{eqn: bijections for tangent space}).  For the reverse direction, a tangent vector $v\colon 
	\mathfrak{m}_x/\mathfrak{m}_x^2\to F$ defines a derivation $\dt\colon \mathcal{O}_{X,x}=\mathfrak{m}_x\oplus F\to F$ in $\operatorname{Der}_{\de,x}(\mathcal{O}_{X,x},F)$ by $f\mapsto v((f-f(x))\!\mod \mathfrak{m}_x^2)+\de(f(x))$.
	
	Let $x'=\psi(x)\in X'(F)$ and define $\de_{x'}\colon F[X']\xrightarrow{\de}F[X']\xrightarrow{x'}F$. Since $\phi\colon G\to G'$ is surjective and the groups $G,G'$ are smooth, it follows from \cite[Chapter II, \S 5, Prop. 5.3]{DG} that the induced map on the Lie algebras is surjective. Since $F$ is assumed to be algebraically closed for now, $X$ and $X'$ are trivial torsors and it follows that the tangent map $T_xX\to T_{x'} X'$ is surjective. Therefore the image of $\de_{x'}$ in $T_{x'} X'$ lifts to a tangent vector in $T_xX$ which corresponds to a derivation $\de_x\colon F[X]\to F$. Chasing through the bijections in (\ref{eqn: bijections for tangent space}) we see that
	$$
	\xymatrix{
		F[X'] \ar^{\psi^*}[rr] \ar_-{\de_{x'}}[rd] & & F[X]  \ar^-{\de_x}[ld] \\
		& F &	
	}
	$$
	commutes. Thus, if we define a differential structure on $X$ via $\de_x$ by virtue of Proposition~\ref{prop: differential structure at point}, it follows from Lemma~\ref{lemma: differential iff differential at point} that $\psi$ is differential. This proves \ref{moving b} in case that $F$ is algebraically closed.
	
	Now let $F$ be arbitrary again. Since $\phi\colon G\to G'$ is surjective we can identify $F[X']$ with a subring of $F[X]$. As $X_\Fa$ is a differential $G_\Fa$-torsor, it follows from what we proved previously that there exists a $G$-invariant derivation
	$\de_{\Fa[X]}\colon \Fa[X]\to \Fa[X]$ that extends the given derivation $\de_{F[X']}\colon F[X']\to F[X']$.
	
	We next show that $\de_{\Fa[X]}\colon \Fa[X]\to \Fa[X]$ restricts to a $G$-invariant derivation $\de_{E[X]}\colon E[X]\to E[X]$, where $E$ is a suitable finite Galois extension of $F$. Namely, suppose that $f_1,\ldots,f_n\in F[X]$ generate $F[X]$ as an $F$-algebra. Then there exists a finite Galois extension $E$ of $F$ such that $\de_{\Fa[X]}(f_1),\ldots,\de_{\Fa[X]}(f_n)\in E[X]$ and thus $\de_{\Fa[X]}$ restricts to a derivation $\de_{E[X]}$ on $E[X]$.
	
For $\tau\in\Galf(E/F)$, the restriction of $\tau(\de_{E[X]})=\tau\circ \de_{E[X]}$ to $F[X]$ is a derivation from $F[X]$ into $E[X]$, where $E[X]$ is considered as an $F[X]$-module via the inclusion $F[X]\subseteq E[X]$. Moreover, $\tau(\de_{E[X]})$ agrees with $\de_{F[X']}$ on $F[X']$, and $\tau(\de_{E[X]})$ is $G$-invariant, since the $G(K)$-action commutes with the $\Galf(E/F)$-action (using here that $K$ is algebraically closed).
	
	This implies that $$\mathcal{R}(\de_{E[X]}):=\frac{1}{|\Galf(E/F)|}\sum_{\tau\in \Galf(E/F)}\tau(\de_{E[X]})$$
	defines a derivation from $F[X]$ into $E[X]$ that agrees with $\de_{F[X']}$ on $F[X']$.
	By definition, an element in the image of $\mathcal{R}(\de_{E[X]})$ is fixed by the $\Galf(E/F)$-action and thus lies in $F[X]$. Therefore $\mathcal{R}(\de_{E[X]})$ actually defines a derivation $\mathcal{R}(\de_{E[X]})\colon F[X]\to F[X]$. As $\tau(\de_{E[X]})$ is $G$-invariant for every $\tau\in\Galf(E/F)$, the derivation $\mathcal{R}(\de_{E[X]})$ is also $G$-invariant and thus defines the desired differential structure on $X$ by Corollary~\ref{cor: invariant derivation}.  Here $\psi$ is a differential morphism over $F$ because it is the restriction of a differential morphism over $\bar F$ (cf.~the algebraically closed case).
\end{proof}

\begin{ex} \label{ex: differential structure on trivial torsor} Let $X= G_F$ be a trivial $G_F$-torsor. Then the derivations $\de\colon F[X]\to F[X]$ that turn $X$ into a differential $G_F$-torsor are in bijection with the Lie algebra of $G_F$.
	 	
To see this, let $x=1\in X(F)\subseteq X(\Fa)$ be the identity element and let $\dt\colon F[X]\to\Fa$ be a derivation extending $\de\colon F\to F$. Since $\tau(x)=x$ for $\tau\in\Gamma$ we see that $\tau(\dt)\colon F[X]\to\Fa$ is a derivation with respect to the $F[X]$-module structure on $\Fa$ given by $f\mapsto f(x)$. From Lemma~\ref{point-derivation-properties}\ref{point lem a} it follows that $\tau(\dt)^x=\tau(\dt)$. Thus $\dt$ is Galois-equivariant if and only if $\tau(\de)=\de$ for all $\tau\in\Gamma$; i.e., $\de(F[X])\subseteq F$. So by Proposition \ref{prop: differential structure at point} the derivations $\de\colon F[X]\to F[X]$ that turn $X$ into a differential $G_F$-torsor are in bijection with the point derivations $\dt\colon F[X]\to F$ at $x$. As in (\ref{eqn: bijections for tangent space}) above, the latter set is in bijection with $T_x X$.
\end{ex}

\begin{ex}\label{Gl_n example}
In particular, consider the group $\GL_n$ and let $X=\GL_{n,F}$ be the trivial $\GL_{n,F}$-torsor, with coordinate ring
$F[X] = F[T,\det(T)^{-1}] := F[t_{ij},\det(T)^{-1}\ |\ 1 \le i,j \le n]$; here 
$T = (t_{ij})$ is an $n\times n$-matrix of indeterminates $t_{ij}$ over $F$.
We can turn $X$ into a differential $\GL_{n,F}$-torsor by defining a derivation $\de_A\colon F[X]\to F[X]$ extending $\de\colon F\to F$ by $\de(T)=AT$ for a fixed $A\in F^{n\times n}$, the Lie algebra of $\GL_{n,F}$. Conversely, if $X$ is a differential $\GL_{n,F}$-torsor, then $X$ is a trivial torsor since $H^1(F,\GL_{n,F})$, which classifies $\GL_{n,F}$-torsors, is trivial by Hilbert's Theorem~90 (e.g., see \cite[Chapter~III, Sect.~1.1, Lemma~1]{Serre:CG}).
It follows from Example \ref{ex: differential structure on trivial torsor} that the derivation on $F[X]$ is of the above form $\de_A$. Thus differential $\GL_{n,F}$-torsors correspond to universal solution rings for linear differential equations $\de(z)=Az$ with $A\in F^{n\times n}$ (and where $z$ is an $n$-tuple of indeterminates).

Distinct choices of $A$ can lead to non-isomorphic differential torsors. For example, if $n=1$, $F={\mathbb C}(x)$ with derivation $d/dx$, then $A=(0)$ defines the trivial differential ${\mathbb G}_m$-torsor (i.e., its constants are ${\mathbb C}[{\mathbb G}_m]$), whereas $A=(1)$ defines a simple differential torsor, with constants~${\mathbb C}$. In particular, a differential torsor can be trivial as a torsor but non-trivial as a differential torsor. 
\end{ex}

\begin{rem}\label{rem: uniqueness of diff}
	Let $H$ be a closed subgroup of $G$ and let $Y$ be a differential $H_F$-torsor. Recall that there is an induced $G_F$-torsor $\Ind_{H_F}^{G_F}(Y)$ (see the appendix, Subsection~\ref{inducedtorsors}). Since $Y\to \Ind_{H_F}^{G_F}(Y)$ is $H_F$-equivariant, it follows from Proposition \ref{prop: moving differential structures}\ref{moving a} that there exists a unique differential structure on $\Ind_{H_F}^{G_F}(Y)$ such that $Y\to \Ind_{H_F}^{G_F}(Y)$ is differential. In the sequel we will always consider $\Ind_{H_F}^{G_F}(Y)$ as a differential $G_F$-torsor by virtue of this differential structure.  
In view of Proposition~\ref{prop: universal property of ind},	
this structure can be made explicit: $\Ind_{H_F}^{G_F}(Y) = (Y \times G)/H = \Spec((F[Y] \otimes_K K[G])^H)$, and the derivation on $(F[Y] \otimes_K K[G])^H$ is the restriction of the one on $F[Y] \otimes_K K[G]$ induced from $F[Y]$ by declaring the elements of $K[G]$ to be constant.
\end{rem}

%-----------------------------------------------------------------------------------------------------
\subsection{Simple differential torsors and Picard-Vessiot rings} \label{subsec PV}
%-----------------------------------------------------------------------------------------------------

As before, $(F,\de)$ is a differential field of characteristic zero with field of constants $K=C_F$, and $G$ is a linear algebraic group over $K$.

Recall that a \textit{Picard-Vessiot ring} over $F$ is a differential ring extension $R/F$ of the form $R=F[Z,\det(Z)^{-1}] := F[z_{ij},\det(Z)^{-1} \ | \ 1 \le i,j \le n]$ for some matrix $Z = (z_{ij}) \in \GL_n(R)$ with $\de(Z)Z^{-1}\in F^{n\times n}$, such that $R$ is a simple differential ring and $C_R=K$. 
(The elements $z_{ij}$ need not be algebraically independent over $F$.)
Equivalently, a Picard-Vessiot ring over $F$ is a differential ring without zero divisors of the form $R=F[Z,\det(Z)^{-1}]$ with $\de(Z)Z^{-1}\in F^{n\times n}$ and such that $C_{\Frac(R)}=K$.
In this situation, we also say that $R$ is a {\em Picard-Vessiot ring for the linear differential equation} $\del(z)=Az$, where $A = \de(Z)Z^{-1}\in F^{n\times n}$ and $z$ is an $n$-tuple of indeterminates.
A \textit{Picard-Vessiot extension} of $F$ is the fraction field of a Picard-Vessiot ring over $F$.
If $E/F$ is an extension of differential fields that is finitely generated as a field extension, then
$E$ is a Picard-Vessiot extension of $F$ if and only if $C_E=C_F$, $E=\Frac(R)$ for some differential ring extension $R$ of $F$, and the left $R$-module $R \otimes_F R$ is generated by its constants (see Definition~1.8 and Theorem~3.11 of \cite{AmanoMasuokaTakeuchi:Hopf}). In this situation, $R$ is the associated Picard-Vessiot ring.

The \textit{differential Galois group} of a Picard-Vessiot ring $R$ is defined as the group functor $\underline{\Aut}^\de(R/F)$.  It is an affine group scheme of finite type over $K$ represented by the $K$-algebra $C_{R\otimes_F R}=K[Z^{-1}\otimes Z, \det(Z^{-1}\otimes Z)^{-1} ]$, where $Z^{-1}\otimes Z$ is a short-hand notation for the matrix product $(Z^{-1}\otimes 1) \cdot (1 \otimes Z)\in \GL_n(R\otimes_F R)$. 
For more details about differential Galois theory, see \cite{vdPutSinger} for the case that the constant field $K$ is algebraically closed; and see \cite{Dyc} and \cite{AmanoMasuokaTakeuchi:Hopf} for the general case.
In particular, there is a differential analog of the usual Galois correspondence; see \cite[Theorem~4.4]{Dyc} 
and \cite[Theorem~2.11]{AmanoMasuokaTakeuchi:Hopf}.  (In the former reference, one must be more careful in defining the invariant subfield $E^H$ of a Picard-Vessiot extension, because $E \otimes_K A$ is not necessarily 
$H$-stable in the total ring of fractions of $R \otimes_K A$, for $A$ a $K$-algebra.  Instead, one can use the definition given at the beginning of Section~3 of \cite{BHH}.)  See also the observation after the proof of 
Proposition~\ref{prop X/N diff} below.

In the above situation, there is a linearization $\underline{\Aut}^\del(R/F) \hookrightarrow\GL_n$ that depends on the choice of a fundamental solution matrix $Z \in \GL_n(R)$. Details can be found in \cite{BHH} and \cite{Dyc}. 
The following proposition explains the relation between Picard-Vessiot rings and differential torsors. 

\begin{prop}\label{prop PVR} 
Let $F,K,G$ be as above. 
\renewcommand{\theenumi}{(\alph{enumi})}
\renewcommand{\labelenumi}{(\alph{enumi})}
\begin{enumerate}
\item \label{PVR a}
Let $R/F$ be a Picard-Vessiot ring with differential Galois group $G$. Then $\Spec(R)$ is a simple differential $G_F$-torsor. 
\item \label{PVR b}
Let $X=\Spec(R)$ be a differential $G_F$-torsor such that $R$ is an integral domain and assume that $C_{\Frac(R)}=K$. Then $R/F$ is a Picard-Vessiot ring with differential Galois group $G$. 
\end{enumerate}
\end{prop} 

\begin{proof} 
\ref{PVR a} \ By definition, $R$ is a simple differential ring. By Kolchin's Theorem, $X=\Spec(R)$ is a $G_F$-torsor. (See \cite[Theorem~1.30]{vdPutSinger} in the case that $K$ is algebraically closed, or \cite[Proposition~2.13]{AmanoMasuokaTakeuchi:Hopf} for the general case.)
Therefore we have isomorphisms $R\otimes_F R\cong R\otimes_K K[G]$ and $C_{R\otimes_F R}\cong K[G]$ (using $C_R=K$). The homomorphism corresponding to the $G_F$-action $X\times G_F\to X$ is 
  $$\rho:R\to R\otimes_F R \xrightarrow{\simeq} R\otimes_K K[G] \xrightarrow{\simeq} R\otimes_F F[G]$$
  where the first map is the inclusion into the second factor. (Cf.~\cite[Lemma 1.9]{AmanoMasuokaTakeuchi:Hopf}.) Clearly this is a differential morphism, and so $X$ is a differential torsor.

\smallskip

\ref{PVR b} \ Let $\rho \colon R\to R\otimes_F F[G] \xrightarrow{\simeq} R\otimes_K K[G]$ be the differential homomorphism corresponding to the $G_F$-action on $X$. 
Since $X$ is a $G_F$-torsor, 
the left $R$-linear extension $R\otimes_ F R\to R\otimes_K K[G]$ is a differential isomorphism; and as $K[G]$ is constant, we see that $R\otimes_ F R$ is generated by constants as a left $R$-module. 
Since $C_{\Frac(R)}=K$ and since $\Frac(R)/F$ is a finitely generated field extension, it follows that
$\Frac(R)/F$ is a Picard-Vessiot extension with Picard-Vessiot ring $R$ (using the equivalent criterion given in \cite[Definition~1.8]{AmanoMasuokaTakeuchi:Hopf}).  Moreover, the differential Galois group of $R/F$ is $G$. Indeed, the isomorphism $R\otimes_F R\to R\otimes_K K[G]$ identifies $C_{R\otimes_F R}$ with $K[G]$, and it is easy to check that this is an isomorphism of Hopf algebras.
\end{proof}

\begin{rem}
The close relationship between Picard-Vessiot rings and
differential torsors explained in Prop. \ref{prop PVR} has a parallel in Kolchin's approach to differential Galois theory in \cite{Kolchin:DifferentialAlgebra}.  There, the Galois groups are not necessarily linear, and the corresponding field
extensions (which are Picard-Vessiot extensions in the case of linear
groups) are called {\em strongly normal}.  Roughly speaking, Theorem~9
of \cite[VI.10]{Kolchin:DifferentialAlgebra} says that there is a bijection between the strongly normal extensions $E$ of $F$ with differential Galois group $G$, and the principal homogeneous $G$-spaces $X$ over $F$ such that $E$ is generated as a differential field over $F$ by an ``$X$-primitive''.
We note that Kolchin's framework uses universal domains rather than
scheme theory, and so his principal homogeneous spaces (see
\cite[V.3]{Kolchin:DifferentialAlgebra}) are not the same as torsors in our sense.
\end{rem}

\begin{cor} \label{cor PV equiv}
Let $F,K,G$ be as above, and let $R/F$ be a differential ring extension.  
\renewcommand{\theenumi}{(\alph{enumi})}
\renewcommand{\labelenumi}{(\alph{enumi})}
\begin{enumerate}
\item \label{PV equiv a}
Suppose that $C_R=K$.  Then $R/F$ is a Picard-Vessiot ring with differential Galois group $G$ if and only if $\Spec(R)$ is a simple differential $G_F$-torsor.
\item \label{PV equiv b}
Suppose that $R$ is an integral domain and that $C_{\Frac(R)}=K$.  Then $R/F$ is a Picard-Vessiot ring with differential Galois group $G$ if and only if $\Spec(R)$ is a differential $G_F$-torsor.
\end{enumerate}
\end{cor}

\begin{proof}
The forward direction of \ref{PV equiv a} was given in Proposition~\ref{prop PVR}\ref{PVR a}, while the reverse direction follows from \ref{PV equiv b} because a simple differential $K$-algebra $R$ is an integral domain that satisfies $C_R = C_{\Frac(R)}$.
The forward direction 
of \ref{PV equiv b} 
follows from Proposition~\ref{prop PVR}\ref{PVR a}, while the reverse direction is immediate from Proposition~\ref{prop PVR}\ref{PVR b}.
\end{proof}

One advantage of working with differential torsors as compared to Picard-Vessiot rings is that for subgroups $H$ of $G$, we can induce differential $H_F$-torsors to differential $G_F$-torsors, which will allow us to apply patching techniques.  See Remark~\ref{rem: uniqueness of diff}.

\begin{prop} \label{prop: torsors are ind}
	Let $X$ be a differential $G_F$-torsor. If $K$ is algebraically closed, then there exists a closed subgroup $H$ of $G$ and a simple differential $H_F$-torsor $Y$ with $C_{\Frac(F[Y])}=K$ such that $X\cong \indF(Y)$ (as differential torsors).  
\end{prop}

\begin{proof}
	Let $I\triangleleft F[X]$ be maximal among the proper differential ideals in $F[X]$ and let $Y\subseteq X$ be defined by $I$. By \cite[Lemma 1.17]{vdPutSinger}), the ideal $I$ is prime, and in particular radical.  
So $F[Y]\cong F[X]/I$, and this is an integral domain.  Moreover, also by \cite[Lemma 1.17]{vdPutSinger}, there are no new constants in the field of fractions $E$ of $F[Y]$.
	
Let $B\subseteq F[Y]\otimes_F F[Y]$ denote the image of $K[G]$ in $F[Y]\otimes_F F[Y]$ under
\begin{equation} \label{eqn: differential map}
F[X]\otimes_K K[G]\cong F[X]\otimes_F F[X]\to F[Y]\otimes_F F[Y].
\end{equation}
Since $F[X]\to F[Y]$ is a surjective differential homomorphism so is the map in (\ref{eqn: differential map}). Therefore the elements of $B$ are constants, and $F[Y]\otimes_F F[Y]$ is generated by $B$ as a left $F[Y]$-module. Hence $E\otimes_F F[Y]$ is generated by $B$ as an $E$-module. It is a general fact that the constants of a differential $E$-algebra are linearly disjoint from $E$ over $C_E$ (\cite[Chapter II, Section 1, Corollary 1, p. 87]{Kolchin:DifferentialAlgebra}). Thus $E\otimes_F F[Y]\cong E\otimes_K B$. It follows that $F[Y]\otimes_F F[Y]\cong F[Y]\otimes_K B$ and that the constants of $F[Y]\otimes_F F[Y]$ equal $B$.
	
	We have seen that $E=\Frac(F[Y])$ satisfies $C_E=K$ and that $F[Y]\otimes_F F[Y]$ is generated by its constants as a left $F[Y]$-module. 
	Thus $E$ is a Picard-Vessiot extension of $F$ by the criterion in \cite[Definition~1.8]{AmanoMasuokaTakeuchi:Hopf},
with Picard-Vessiot ring $F[Y]$.	
By \cite[Lemma~1.9]{AmanoMasuokaTakeuchi:Hopf} 
and the discussion on pages 137-138 there,
the associated differential Galois group is $H:=\Spec(B)$.  Hence $Y$ is 
an $H_F$-torsor, by Corollary~\ref{cor PV equiv}\ref{PV equiv b}.
Since $K[G]\to B=K[H]$ is surjective, $H$ is a closed subgroup of $G$. 
	
	We have a commutative diagram:
	\begin{equation} \label{eq: action compatible}
	\xymatrix{
F[X]\otimes_F F[X] \ar[d] \ar^{\cong}[r]& F[X]\otimes_K K[G] \ar[d] \\
	F[Y]\otimes_F F[Y]  \ar^{\cong}[r] & 	F[Y]\otimes_K K[H]
	}
	\end{equation}
	The morphisms dual to the action of $G_F$ and $H_F$ on $X$ and $Y$, respectively, are obtained from the horizontal isomorphisms in (\ref{eq: action compatible}) by precomposing with the inclusions into the second factor. Thus the commutativity of (\ref{eq: action compatible}) shows that the inclusion morphism $Y\to X$ is $H_F$-equivariant. 
	Thus $X$ and $\indF(Y)$ are isomorphic as $G_F$-torsors by Remark~\ref{rem: uniqueness of ind}\ref{ind rem a}, and as differential torsors by Remark~\ref{rem: uniqueness of diff}.
\end{proof}

\begin{lem}	\label{lem: simple not ind}
	Let $H$ be a closed subgroup of the linear algebraic group $G$, and let $Y$ be a differential $H_F$-torsor. If $\iind_{H_F}^{G_F}(Y)$ is a simple differential $G_F$-torsor, then $H=G$.
\end{lem}

\begin{proof}
	Since $Y\to X:=\iind_{H_F}^{G_F}(Y)$ is differential, the morphism $F[X]\to F[Y]$ is differential. It is surjective by Lemma \ref{lem: canonical map is embedding}, and injective because $F[X]$ is simple.  Thus $F[X]\cong F[Y].$ This implies $F[X]\otimes_K K[G]\cong F[Y]\otimes_K K[H]$ and therefore $H=G$.
\end{proof}
 
 The following proposition concerns the passage from a linear algebraic group $G$ to its quotient by a normal subgroup~$N$. See Subsection~\ref{quotients} for a discussion of quotients of torsors.

\begin{prop}\label{prop X/N diff}
	Let $N\trianglelefteq G$ be a normal closed subgroup. If $X$ is a differential $G_F$-torsor, then $Y:=X/N$ is a differential $(G/N)_F$-torsor. Moreover, if $X$ is a simple differential $G_F$-torsor, then $Y$ is a simple differential $(G/N)_F$-torsor.
\end{prop}

\begin{proof}
First recall that $X/N$ exists as an affine variety and is a $(G/N)_F$-torsor by
Proposition \ref{prop X/N}\ref{normal part}; viz.\ $X/N = \widetilde{X\qq N_F}=\Spec(F[X]^{N_F})$.
Applying Proposition~\ref{prop: moving differential structures}\ref{moving a} to the quotient morphism
$\psi:X \to Y =X/N$ yields a unique differential structure on $Y$ that is compatible with that of $X$.
	
Next, assume that $X$ is a simple differential $G_F$-torsor; i.e., that $F[X]$ has no non-trivial differential ideals. Let $I \trianglelefteq F[X]^{N_F}$ be a non-zero differential ideal. Then $J=I\cdot F[X]$ is a non-zero differential ideal of $F[X]$, hence $J=F[X]$. By Proposition \ref{prop X/N}\ref{normal part}, $F[X]$ is faithfully flat over $F[X]^{N_F}$ and thus $I=J\cap F[X]^{N_F}$ (see \cite[Thm.~7.5.(ii)]{Mats}). Hence $I=F[X]^{N_F}$.  Thus $F[X]^{N_F}$ is differentially simple, and so $Y$ is a simple differential $(G/N)_F$-torsor. 
\end{proof}

Observe that Corollary~\ref{cor PV equiv} and Proposition~\ref{prop X/N diff} provide a partial version of the Galois correspondence in differential Galois theory, via torsors.  Namely, if $R$ is a Picard-Vessiot ring over $F$ with differential Galois group $G$, and if $N$ is a closed normal subgroup of $G$, then $R^N$ is a Picard-Vessiot ring over $F$ with differential Galois group $G/N$, by considering the corresponding torsors $X = \Spec(R)$ and $X/N = \Spec(R^N)$.  Also, if $X=\Spec(R)$ is as above and $H$ is any closed subgroup of $G$ such that $F(X/H)=F$, then $H=G$.  To see this, first recall that $X/H$ is an integral quasi-projective $F$-variety (by Proposition~\ref{prop X/N}\ref{X/H part} and the integrality of $X$); so if its function field is $F$ then it is isomorphic to a single $F$-point $\Spec(F)$.  Thus $\Spec(\Fa)=(X/H)(\Fa) = X(\Fa)/H(\Fa) \cong G(\Fa)/H(\Fa)$; i.e.\ $H(\Fa)=G(\Fa)$, and hence $H=G$ since $\Fa$-points are dense.  The above facts will be useful later.

\begin{lem}  \label{lem theta}	
Let $H$ be a closed subgroup of the linear algebraic group $G$, let $Y$ be a differential $H_F$-torsor, 
and write $X= \Ind_{H_F}^{G_F}(Y)$. Then
$$\psi\colon Y\times G_F\to Y\times X,\ (y,g)\mapsto (y,y.g)$$
is an isomorphism. If $L$ is a differential field containing $R := F[Y]$ as a differential subring, the pullback of $\psi$ via $\Spec(L)\to Y$ yields an isomorphism of differential $G_L$-torsors $G_L\to X_L$. 
With $S := F[X]= \Ind_{H_F}^{G_F}(R)$,
the dual isomorphism on coordinate rings
$$\Theta: L[X]=L \otimes_F S=L\otimes_F(R\otimes_F F[G])^H\longrightarrow^{\!\!\!\!\!\!\!\!\!\!\sim\,}\  L[G]$$
is given by $\Theta(\sum_i a_i \otimes r_i \otimes f_i)=\sum_i a_i \cdot r_i \cdot f_i$, for all elements $a_i \in L \subseteq L[G], r_i \in R \subseteq L \subseteq L[G], f_i \in F[G] \subseteq L[G]$ such that $\sum_i a_i \otimes r_i \otimes f_i\in L\otimes_F(R\otimes_F F[G])^H$. 		
\end{lem}

\begin{proof}
	Since $X=\indF(Y)$ is a $G_F$-torsor,
	the morphism $X \times G_F \to X \times X$ given by $(x,g) \mapsto (x,x.g)$ is an isomorphism.  Restricting to 
	$Y \times G_F$ yields the isomorphism $\psi$.
As $\psi$ is $G_F$-equivariant (where $G_F$ is acting on the right factors), $G_L\to X_L$ is an isomorphism of $G_L$-torsors. The inclusion $(F[Y]\otimes_F F[G])^H\subseteq F[Y]\otimes_F F[G]$ corresponds to the morphism $Y\times G_F\to\indF(Y),\ (y,g)\mapsto y.g$, and so the dual map of $\psi$ is given by $$F[Y]\otimes_F(F[Y]\otimes_F F[G])^H\longrightarrow F[Y]\otimes_F F[G],\ \sum_i a_i \otimes r_i \otimes f_i\mapsto\sum_i a_i \cdot r_i \otimes f_i$$
	for all $a_i \in F[Y], r_i \in F[Y], f_i \in F[G]$ such that $\sum_i a_i \otimes r_i \otimes f_i\in F[Y]\otimes_F(F[Y]\otimes_F F[G])^H$.  Tensoring over $F[Y]$ with $L$ yields the last assertion.
\end{proof}

%----------------------------------------------------------------------------------------------------------------------------------------------------------------------------------------------------------
\section{Differential patching and embedding problems} \label{sec: Patch Embed}
%----------------------------------------------------------------------------------------------------------------------------------------------------------------------------------------------------------

In this section, we construct differential torsors over a given field $F$ by using  differential torsors over larger fields, and patching them.  This builds on the method of patching over fields (see \cite{HH:FP}), and in particular on a result in \cite{HHKtorsor} on patching torsors.  The new aspect in our situation is the differential structure on the torsors.  Using this approach and the correspondence between simple differential torsors and Picard-Vessiot extensions from the previous section, we can construct Picard-Vessiot extensions with desired properties.  This is useful both in the inverse differential Galois problem and for solving embedding problems in differential Galois theory.

%-----------------------------------------------------------------------------------------------------
\subsection{Patching differential torsors}
%-----------------------------------------------------------------------------------------------------

The basic situation is the following:  We have a quadruple of fields $(F,F_1,F_2,F_0)$ together with inclusions $F \hookrightarrow F_i \hookrightarrow F_0$ for $i=1,2$, such that the diagram
\[\xymatrix{
& F_0 & \\
 F_1  \ar[ur]   && F_2 \ar[ul] \\
            & F \ar[ul] \ar[ur]  & } \] 
commutes and such that $F$ is the intersection of $F_1$ and $F_2$ taken inside $F_0$.  (Thus $F$ is the inverse limit of the finite inverse system consisting of the fields $F_i$.)
We refer to such a quadruple as a {\it diamond}.  We say that a diamond has the {\it factorization property} if for every $n>0$, every element          
$A \in \GL_n(F_0)$ can be written as $A_2^{-1}A_1$ with $A_i \in \GL_n(F_i)$.  

It has been useful in applications (e.g. in 
Section 9 of \cite{HHKtorsor}) to consider more general collections of fields and 
inclusions; and the applications in that more general situation do not seem to follow easily from the case of diamonds of fields.
More precisely, as in \cite{HHKtorsor}, a 
\textit{factorization inverse system} over a field $F$ is a finite inverse system $\{F_i\}_{i\in \I}$ of fields whose inverse limit (in the category of rings) is $F$, whose index set $\I$ has a partition $\I = \I_v \cup \I_e$ into a disjoint union such that for each index $k \in \I_e$ there are exactly two elements $i, j \in \I_v$ for which the inverse system contains maps $F_i,F_j \to F_k$, and such that there are no other maps given in the inverse system.  If there is a map $F_i \to F_k$ in the inverse system, we write $i \succ k$; this defines a partial ordering on $\I$.  

A factorization inverse system determines a (multi-)graph $\Gamma$ whose vertices are the elements of $\I_v$ and whose edges are the elements of $\I_e$. The vertices of an edge $k \in \I_e$ correspond to the elements $i, j \in \I_v$ such that $i, j \succ k$. Note that the graph $\Gamma$ is connected (otherwise, the inverse limit $F$ would have zero divisors). For every $k\in \I_e$, we fix a labeling $l=l_k$ and $r=r_k$ of its vertices $l$ and $r$ 
(i.e., we assign each edge a left vertex and a right vertex). Note that an element $i \in \I_v$ can be a left vertex of an edge and a right vertex of another edge. When working with factorization inverse systems, we always assume that such an orientation of the edges has been fixed. A factorization inverse system $\{F_i\}_{i\in \I}$ has the \textit{simultaneous factorization property} if for any collection of matrices $A_k \in \GL_n(F_k)$, for $k \in \I_e$, there exist matrices $A_i \in \GL_n(F_i)$ for all $i \in \I_v$ such that $A_k=A_{r_k}^{-1}\cdot A_{l_k}$ for all $k \in \I_e$, where we view $A_{r_k}$ and $A_{l_k}$ as matrices with entries in $F_k$ via the inclusions $F_{r_k}, F_{l_k} \hookrightarrow F_k$.  In the case that $\I_e$ has just one element $0$, and $\I_v$ has two elements $l_k=1, r_k=2$, we recover the notions of a diamond and of a diamond with the factorization property.

Turning now to our situation, a
\textit{differential factorization inverse system} over a field $F$ is a factorization inverse system over $F$, such that all fields $F_i$, $i \in \I$, are differential fields of characteristic zero and such that the inclusions $F_{l_k}, F_{r_k} \hookrightarrow F_k$ are differential homomorphisms for all $k \in \I_e$. Note that then $F$ inherits a structure as a differential field (of characteristic zero) and the embeddings $F \hookrightarrow F_i$ given by mapping to the $i$-th component in the inverse limit are differential homomorphisms for all $i \in \I$.  In the case of a diamond, we call the quadruple a \textit{differential diamond}.  A differential factorization inverse system (and in particular a differential diamond) may have the factorization property, as defined above.

\begin{ex} \label{ex diamond}
\renewcommand{\theenumi}{(\alph{enumi})}
\renewcommand{\labelenumi}{(\alph{enumi})}
\begin{enumerate}
\item \label{ex diamond cx}
Let $F = \C(x)$, the field of rational functions over $\C$, or equivalently the field of meromorphic functions on the Riemann sphere $\P^1_\C$.  Choose open discs $U_1$ centered at the point $x=0$ and $U_2$ centered at $x=\infty$, such that $U_1 \cup U_2 = \P^1_\C$.  Let $U_0$ be the annulus $U_1 \cap U_2$, and write $F_i$ for the field of meromorphic functions on $U_i$ for $0,1,2$.  Then $(F,F_1,F_2,F_0)$ is a differential diamond with the factorization property with respect to the derivation $d/dx$.  See Lemma~\ref{int-fact lem} for a more general statement concerning open subsets of Riemann surfaces.  This will be used in solving differential embedding problems over complex function fields.
\item \label{ex diamond other}
There are also examples of 
factorization inverse systems with the simultaneous factorization 
property over function fields over a complete discretely valued field 
$K$; see \cite[Corollary~3.4, Proposition~2.2]{HHKtorsor}.  In the case of the 
function field $K(x)$ together with the derivation $d/dx$ in 
characteristic zero, we obtain differential factorization inverse 
systems.  Such examples can be viewed as rigid analytic analogs of 
Example~\ref{ex diamond}\ref{ex diamond cx}.
\end{enumerate}
\end{ex}

Following \cite[Section~2]{HHKtorsor}, for a linear algebraic group $G$ over $F$ and a factorization inverse system of fields $\{F_i\}_{i \in \I}$ with inverse limit $F$, we define a $G$-{\it torsor patching problem} to be a system of $G_{F_i}$-torsors $X_i=\Spec(S_i)$ together with $F_k$-isomorphisms of $G_{F_k}$-torsors $\nu_{ik}:F_k \times_{F_i} X_i \to X_k$ for all pairs of distinct indices with $i \succ k$ (i.e., such that $F_i \subseteq F_k$).  Here $\nu_{ik}$ corresponds to an isomorphism of coordinate rings $\nu_{ik}^*:S_k \to F_k\otimes_{F_i} S_i$ that respects the $G$-actions.  A {\it solution} to the patching problem is a torsor over $F$ that induces the torsors $X_i$ compatibly via base change.  That is, a solution is given by a $G$-torsor $X = \Spec(S)$ over $F$
together with $F_i$-isomorphisms of $G_{F_i}$-torsors $\gamma_i:F_i \times_F X \to X_i$ for all $i \in \I_v$ such that 
the two maps $\nu_{ik}\circ (\id_{F_k} \otimes_{F_i} \gamma_i)  : F_k \times_F X\to X_k$ agree, for $i=l_k,r_k$.
Patching problems and solutions can also be described
on the level of coordinate rings.  If we write 
$\Theta_{ki} = (\nu_{ik}^*)^{-1}:F_k\otimes_{F_i} S_i \to S_k$ and
$\Phi_i = (\gamma_i^{-1})^*:F_i\otimes_F S \to S_i$, the compatibility condition is that the two isomorphisms
$\Theta_{kl_k}\circ (\Id_{F_k}\otimes_{F_{l_k}} \Phi_{l_k}), \, \Theta_{kr_k}\circ(\Id_{F_k}\otimes_{F_{r_k}} \Phi_{r_k}):F_k\otimes_{F}S \to S_k$ coincide.  

It was shown at \cite[Proposition~2.2, Theorem~2.3]{HHKtorsor}
that if the factorization inverse system above has the simultaneous factorization property, then up to isomorphism there is a unique solution $X=\Spec(S)$, given on the level of coordinate rings by the inverse limit.  That is, $\displaystyle S = \lim_\leftarrow S_i$, where the limit is over all $i \in \I$; and $\Phi_i$ is induced by the canonical map $S \to S_i$ for $i \in \I_v$.  Thus for $j \in \I_v$, the map $\Phi_j$ sends $a \otimes (x_i)_{i \in \I} \in F_j \otimes_F S$ to $ax_j$ in $S_j$.  Also, since each $\gamma_i$ is an isomorphism of torsors, the co-action homomorphism
$\rho:  S\to S \otimes_F F[G]$ corresponding to the $G$-action on $X$ is the restriction of the product map 
$(\prod \rho_i) \colon \prod S_i \to (\prod S_i)\otimes_F F[G]$, where 
$\rho_i \colon S_i\to S_i \otimes_{F_i} F_i[G]\cong S_i \otimes_{F} F[G]$ is the co-action map corresponding to the $G_{F_i}$-action on $\Spec(S_i)$ for $i \in \I$.

The notions of patching problems and solutions carry over to the differential situation.  Namely, 
consider a differential factorization inverse system $\{F_i\}_{i \in \I}$ over a differential field $F$ of characteristic zero, and let $G$ be a linear algebraic group over $K:= C_F$.  A system of differential $G_{F_i}$-torsors $X_i$ such that the maps $\nu_{ik}$ as above are differential isomorphisms will be called a
{\it patching problem of differential $G_F$-torsors}.  
(Recall that in our setup, $G$ is defined over the constants $K$ of $F$.  Also, recall that the derivation on $F_i[G]$ is defined to extend the given derivation on $F_i$ and to be constant on $K[G]$.)
Similarly, if $X$ is a differential $G_F$-torsor and each $\gamma_i$ as above is an isomorphism of differential torsors, we have a {\it solution} to this differential patching problem.
The result from \cite{HHKtorsor} cited above then carries over to this situation, since the solution given by \cite{HHKtorsor} inherits a unique compatible derivation on the coordinate ring by restriction:

\begin{thm} \label{thm diff solution} Let $\{F_i\}_{i \in \I}$ be a differential factorization inverse system over $F$ with the simultaneous factorization property and let $K=C_F$ be the field of constants of $F$. Let $G$ be a linear algebraic group over $K$ and 
let $(\{S_i \}_{i \in \I}, \{\Theta_{kl_k}, \Theta_{kr_k}\}_{k \in \I_e})$ define a patching problem of differential $G_F$-torsors. Then up to differential isomorphism, there exists a unique solution $(S, \{\Phi_i\}_{i \in \I_v})$, given by $\displaystyle S=\lim_\leftarrow S_i$, with $\Phi_i$ induced by the natural map $S \to S_i$, and with the $G$-action and derivation on $S$ given by restriction from those on the rings $S_i$. 
\end{thm}

In particular, this holds in the case of differential diamonds $(F,F_1,F_2,F_0)$ with the factorization property.  There, we may identify $S$ with the intersection $\Theta_{01}(S_1) \cap \Theta_{02}(S_2) \subseteq S_0$. With this identification, the isomorphism $\Phi_i:F_i \otimes_F S \to S_i$ sends $ a\otimes s$ to $a\cdot \Theta_{0i}^{-1}(s)$ for $i =1,2$, for $a \in F_i$ and $s \in S$.

%-----------------------------------------------------------------------------------------------------
\subsection{Patching Picard-Vessiot rings}\label{subsec PVR}
%-----------------------------------------------------------------------------------------------------

Using the above theorem, we prove a result about patching Picard-Vessiot rings, by relying on the relationship between Picard-Vessiot rings and differential torsors.  This enables us to construct Picard-Vessiot extensions of $F$ with a given group $G$, by using Picard-Vessiot extensions of appropriate overfields $F_i$ with differential Galois groups $G_i$ that generate $G$.

\begin{lem}\label{lem ideal cor}
 Let $R$ be a simple differential ring with field of constants $K$ and let $A$ be a $K$-algebra. We consider $A$ as a constant differential ring. Then there is a bijection between the differential ideals in $R\otimes_K A$ and the ideals in $A$, given by $I \mapsto I \cap A$ for differential ideals $I \trianglelefteq R\otimes_K A$, with inverse
$J \mapsto R \otimes_K J$ for ideals $J \trianglelefteq A$. 
\end{lem}

\begin{proof}
 This is a well-known statement in differential algebra and it follows as in \cite[Prop. 5.6]{Kovacic2}. See also \cite[Lemma 10.7]{Maurischat}. 
\end{proof}

In the following, we use the notation $\Ind_H^G(R)$ for the coordinate ring of $\Ind_H^G(\Spec(R))$, when $H$ is a subgroup of $G$ and $\Spec(R)$ is an $H$-torsor. By the comment before Remark \ref{rem ind GG}, $\Ind_H^G(R)$ is the ring of invariants $(R\otimes_K K[G])^H$ with respect to a certain $H$-action on $\Spec(R)\times G$. A definition of the ring of invariants can be found in the second paragraph of Section \ref{subsec Gspaces}. For a $G$-space $\Spec(R)$ (and in particular for $G$ itself) we also use the notion of $H$-stable ideals in $R$ for subgroups $H$ of $G$ (see Lemma \ref{lemma:invariant ideal} and the subsequent paragraph).

\begin{lem} \label{lem patching}
Let $F \subseteq F_1 \subseteq F_0$ be extensions of differential fields, 
and let $G$ be a linear algebraic group over $K := C_F$.
Let $S$ be a differential ring containing $F$, and for $i=0,1$ write
$S_i = F_i \otimes_F S$.  Let $H_1$ be a closed subgroup of $G$, and
suppose that 
$S_1 = \Ind_{H_1}^G(R_1)$ for a Picard-Vessiot ring $R_1 \subseteq F_0$ over $F_1$
with differential Galois group $(H_1)_{C_{F_1}}$.
Let $I$ be a differential ideal of $S$, and 
for $i=0,1$ write $I_i := F_i \otimes_F I \subseteq S_i$.  View 
$I_0$ as an ideal in $F_0[G]$ via the 
isomorphism $\Theta:S_0 = F_0 \otimes_{F_1} \Ind_{H_1}^G(R_1) \to F_0[G]$ given in 
Lemma~\ref{lem theta}.   
For $i=0,1$, consider the right action of $G_{F_i}$ on itself given by $(g',g) \mapsto g^{-1}g'$.
Then
\renewcommand{\theenumi}{(\alph{enumi})}
\renewcommand{\labelenumi}{(\alph{enumi})}
\begin{enumerate}
\item \label{lem stable H}
For every closed subgroup $H \subseteq G$, the ideal $I_0 \subseteq F_0[G]$ is $H$-stable if and only if
$I_0 \cap F_1[G] \subseteq F_1[G]$ is $H$-stable.
\item \label{lem stable H1}
The ideals $I_0 \subseteq F_0[G]$ and $I_0 \cap F_1[G] \subseteq F_1[G]$ are $H_1$-stable.
\end{enumerate}
\end{lem}

\begin{proof}
The forward direction of part~\ref{lem stable H}
is clear. For the converse direction, it suffices to show that $(I_0\cap F_1[G])F_0[G]=I_0$.

Let $\vartheta \colon R_1\otimes_{F_1} F_1[G]\to F_0[G]$ be the differential homomorphism 
induced by the inclusions of $R_1$ and $F_1$ into $F_0$.
By Lemma~\ref{lem theta}, $\Theta|_{S_1}=\vartheta|_{S_1}$.
Since $\Ind_{H_1}^G(R_1) = (R_1 \otimes_{F_1} F_1[G])^{H_1}$ by
Proposition~\ref{prop: universal property of ind},
we have the following commutative diagram of differential $F$-algebras, where 
$\Phi_i:S \to S_i$ is the natural map for $i=0,1$:
\[
\xymatrix  {
F_0[G] 
& F_1[G] \ar@{_{(}->}[l] \ar@{^{(}->}[d]_\iota \\
S_0 \ar[u]^{\wr}_\Theta
& R_1 \otimes_{F_1} F_1[G] 
\ar[ul]_-{\vartheta} 
\\
S \ar[u]_{\Phi_0} \ar[r]^-{\Phi_1} & S_1 = (R_1 \otimes_{F_1} F_1[G])^{H_1} 
\ar@{_{(}->}[ul]
\ar@{_{(}->}[u] 
}
\] 
Note that $I_i$ is the ideal of $S_i$ generated by $\Phi_i(I)$, for $i=0,1$.
 Write $\tilde J_1$ for the ideal of $R_1 \otimes_{F_1} F_1[G]$ generated by $I_1$.  
If we identify $S_0$ with $F_0[G]$ via $\Theta$, then $I_0$ is identified with the 
ideal of $F_0[G]$ generated by $\Theta(I_1)$, or equivalently 
the ideal generated by $\vartheta(\tilde J_1)$. 

Let $K_1=C_{F_1}$.  As $R_1$ is differentially simple, every ideal in $R_1\otimes_{F_1} F_1[G]\cong R_1\otimes_{K_1} K_1[G]$ is generated by its intersection with $K_1[G]$ by Lemma~\ref{lem ideal cor};
and in particular it is generated by its intersection with $F_1[G]$. Thus $J_1:=\tilde J_1\cap {F_1}[G]$ generates $\tilde J_1$ as an ideal of $R_1\otimes_{F_1} F_1[G]$.
But $\vartheta(\tilde J_1)$ generates $I_0$ as an ideal of $F_0[G]$; hence 
$I_0=J_1F_0[G]$.
Also, $J_1=I_0\cap F_1[G]$ since 
$F_0[G]\cong F_0\otimes_{F_1} F_1[G]$ is faithfully flat over $F_1[G]$
(see \cite[Theorem~7.5(b)]{Mats}).  Thus 
$(I_0\cap F_1[G])F_0[G]=I_0$, proving part~\ref{lem stable H}.

The ideal $\tilde J_1 \subseteq R_1 \otimes_{F_1} F_1[G]$ is $H_1$-stable
(with respect to the action given in Proposition~\ref{prop: universal property of ind})
since it is the extension of the ideal $I_1 \subseteq S_1 = (R_1 \otimes_{F_1} F_1[G])^{H_1}$.
Since the projection morphism $\iota^*:(\Spec R_1) \times_{F_1} G_{F_1} \to G_{F_1}$ is $H_1$-equivariant,
and since $\tilde J_1$ is the extension of $J_1 \subseteq F_1[G]$ with respect to $\iota$, 
it follows that $J_1$ is $H_1$-stable (by stability condition~\ref{stable cond b} of 
Lemma~\ref{lemma:invariant ideal}).    
The $H_1$-stability of $J_1 = I_0 \cap F_1[G] \subseteq F_1[G]$
implies the $H_1$-stability of 
$I_0 \subseteq F_0[G]$, by part~\ref{lem stable H}.
This proves part~\ref{lem stable H1}.
\end{proof}

\begin{thm}\label{thm patching PVR}
In the context of Theorem~\ref{thm diff solution}, suppose that $G$ is generated by a set of closed $K$-subgroups $H_i$, for $i\in \I_v$; and that for each $i \in \I_v$ there is a Picard-Vessiot ring $R_i/F_i$ with differential Galois group $(H_i)_{K_i}$ such that $S_i = \Ind_{H_i}^G(R_i)$, where $K_i = C_{F_i}$.  Suppose also that $S_k = F_k[G]$ for all $k \in \I_e$; and suppose that for each $i \succ k$ there is an embedding $R_i \hookrightarrow F_k$ of differential rings. Suppose moreover that the map $\Theta_{ki} \colon F_{k}\otimes_{F_i}S_i\to S_k$ is the isomorphism of differential $G_{F_{k}}$-torsors given in 
Lemma~\ref{lem theta} with respect to that embedding (for all $i$).  Then $S/F$ is a Picard-Vessiot ring with differential Galois group $G$.
\end{thm}

\begin{proof} 
By Theorem~\ref{thm diff solution}, $\Spec(S)$ is a differential $G_F$-torsor and $S$ is the inverse limit of $\{S_i\}_{i \in \I}$ with respect to the maps $\Theta_{ki}$.
As in Theorem~\ref{thm diff solution}, we write $\Phi_i:S \to S_i$ for
the projection onto the $i$-th component. 
We identify $S$ with $F \otimes_F S \subseteq F_i \otimes_F S \subseteq F_k \otimes_F S$, and similarly identify $S_i$ with $F_i \otimes_{F_i} S_i \subseteq F_k \otimes_{F_i} S_i$.  By Proposition \ref{prop: universal property of ind}, 
$S_i=\Ind_{H_i}^G(R_i)$ consists of the invariants of 
$R_i\otimes_{F_i} F_i[G]$ under the $(H_i)_{F_i}$-action on $\Spec(R_i)\times G_{F_i}$ given by the formula $(x,g).h=(x.h, h^{-1}g)$.  This action restricts to an $H_i$-action on $G$ given by $g.h=h^{-1}g$, which is also the restriction of the (right) $G$-action on $G$ given by $g.g'=g'^{-1}g$.  

To prove the theorem, it suffices to show that $C_S=K$ and that $S$ is differentially simple, by 
Corollary~\ref{cor PV equiv}\ref{PV equiv a}.

\smallskip 

\textit{First step:} We show $C_S=K$. 

Let $x\in S$ be constant, and write $x=(x_i)_{i\in \I}$ with $x_i \in S_i$ constant.  
We wish to show that $x \in K$.
For $i \in \I_v$, the constants of $R_i\otimes_{F_i} F_i[G]=R_i\otimes_{K_i} K_i[G]$ equal $K_i[G]$, since $C_{R_i}=K_i$;
and thus $C_{S_i}=K_i[G]^{(H_i)_{K_i}}\subseteq F_i[G]^{(H_i)_{F_i}}$. 
Hence $x_i \in F_i[G]^{(H_i)_{F_i}} \subseteq F_i[G]$.  Suppose that $i \succ k$.  Then $x_k$ is the image of $x_i$ under the natural inclusion 
$F_i[G] \to F_k[G]$ (see Lemma~\ref{lem theta}).  The co-action on $F_i[G]$ is the restriction
of the co-action on $F_k[G]$; so
$x_i$ is invariant under a given subgroup 
$H \subseteq G$ if and only if $x_k$ is invariant under $H$.  Since the graph $\Gamma$ associated to $\{F_i\}_{i \in \I}$ is connected, it follows that all the $x_i$ (for $i \in \I$) are invariant under the same subgroups of $G$.  As a consequence, each of these elements is invariant under $H_j$ for every $j \in \I_v$, since $x_j$ is.

Thus $x_i \in \bigcap \limits_{j \in \I_v} F_i[G]^{(H_j)_{F_i}}$ for all $i \in \I_v$, and this intersection equals $F_i$ by Lemma~\ref{lem invariants}\ref{stable ideals a}. Hence $x=(x_i)_{i \in \I} \in \varprojlim\limits_{i \in \I} F_i=F$ and thus $x \in C_F=K$.

\smallskip

\textit{Second step:} We show that $S$ is differentially simple.
Let $I$ be a proper differential ideal of $S$.  It suffices to show that $I = (0)$.  

For an edge $k \in \I_e$ and a vertex $i \in \I_v$ of $k$, 
let $I_i$ be the ideal of $S_i$ generated by $\Phi_i(I)$, and let $I_k$ be the ideal 
of $S_k = F_k[G]$ generated by $\Theta_{ki}(I_i)$ (which is independent of the choice of vertex $i$ of $k$).
Let $J_i = I_k \cap F_i[G]$.  We may now apply
Lemma~\ref{lem patching}, with $F_i,F_k$ playing the roles of $F_1,F_0$, and where we consider
the right action of $G$ on itself given by  $(g',g) \mapsto g^{-1}g'$.
By part~\ref{lem stable H} of the lemma, for any subgroup $H$ of $G$, $J_i$ is $H$-stable under 
this action if and only if $I_k$ is.  Since this holds for all such pairs $(k,i)$, and since the graph is connected, it follows that all the ideals $J_i \trianglelefteq F_i[G]$ (for $i \in \I_v$) and $I_k \trianglelefteq F_k[G]$ (for $k \in \I_e$) are stabilized by the same subgroups of $G$.  
But by part~\ref{lem stable H1} of the lemma, 
$J_i$ and $I_k$ are $H_i$-stable, with respect to the above action.
Thus for {\em every} $k' \in \I_e$, the ideal $I_{k'} \trianglelefteq F_{k_0}[G]$ is stable under $H_i$ for all $i \in \I_v$.  But the subgroups $H_i$ generate $G$.  So $I_{k'}$ is $G_{F_{k'}}$-stable by Lemma~\ref{lem invariants}\ref{stable ideals b} and thus $I_{k'}=(0)$ by Lemma~\ref{lem invariants}\ref{stable ideals c}.  Hence $I \subseteq I_{k'}$ is also the zero ideal, as asserted.
\end{proof}

As an example application, we show how Theorem \ref{thm patching PVR} can be applied to show that $\SL_2$ is a differential Galois group over $F$ under the assumption that there exists a differential diamond $(F,F_1,F_2,F_0)$ with the factorization property such that $F_0$ contains ``logarithmic elements'' over $F_1$ and $F_2$. 

\begin{ex} \label{ex SL2} Let $(F,F_1,F_2,F_0)$ be a differential diamond with the factorization property. Assume that there exist Picard-Vessiot rings $R_1/F_1$ and $R_2/F_2$ 
with differential Galois groups the additive group $\Ga$ and such that $R_1\subseteq F_0$ and $R_2\subseteq F_0$. Then there is a Picard-Vessiot ring $S/F$ with differential Galois group $\SL_2$. Indeed, let $H_1$ and $H_2$ be the subgroups of $\SL_2$ of upper and lower triangular matrices whose diagonal entries are equal to $1$. These are known to generate $\SL_2$ over any field. The groups $H_1$ and $H_2$ are isomorphic to $\Ga$ and the Tannaka formalism implies that there are $(2\times 2)$-fundamental solution matrices for $R_1/F_1$ and $R_2/F_2$ such that the corresponding representations of the differential Galois groups of $R_1/F_1$ and $R_2/F_2$ yield $H_1$ and $H_2$ (see \cite[Prop. 3.2]{BHH}). 
Since $\SL_2$ is generated by $H_1$ and $H_2$, Theorem \ref{thm patching PVR} implies that there exists a Picard-Vessiot ring $S/F$ with differential Galois group $\SL_2$.

Note that the assumptions on the existence of $R_1/F_1$ and $R_2/F_2$ are equivalent to the existence of ``logarithmic elements'' $y_1, y_2 \in F_0$ such that $y_1 \notin F_1$ and $\del(y_1) \in F_1$ and similarly $y_2 \notin F_2$ and $\del(y_2) \in F_2$. 
\end{ex}

\begin{rem} 
If $F=k((t))(x)$ for some field $k$ of characteristic zero and $\del=d/dx$, a weaker version of Theorem \ref{thm patching PVR} was proven in \cite[Thm 2.4(a)]{BHH} using ad-hoc methods, on the way to solving the inverse differential Galois problem over $F$. However, that theorem only applies to factorization inverse systems $\{F_i\}_{i \in \I}$ over $F$ where the corresponding graph $\Gamma$ is star-shaped, where all fields of constants $C_{F_i}$ equal $C_F$ and where the Picard-Vessiot ring over $F_i$ is trivial for the internal node $i$ of the star. Theorem \ref{thm patching PVR} does not rely on any of these assumptions 
and it can also be applied to more general factorization inverse systems, e.g. of the sort 
that arise in \cite{HHKtorsor}. 
\end{rem}

%-----------------------------------------------------------------------------------------------------
\subsection{Embedding problems} \label{subsec EP}
%-----------------------------------------------------------------------------------------------------

As in ordinary Galois theory, one can consider embedding problems in differential Galois theory.  Using the above ideas, we prove a result about split differential embedding problems. 

Let $F$ be a differential field with field of constants $K$.  In analogy with the case of
ordinary Galois theory,  
a {\em differential embedding problem} over $F$ consists of an epimorphism of linear algebraic groups $G \to H$ over $K$, say with kernel $N$, together with
a Picard-Vessiot ring $R/F$ with differential Galois group $H$.  (Notice that we work with Picard-Vessiot rings here rather than Picard-Vessiot extensions, because those rings are needed to define the differential Galois groups as group schemes.) 
If the short exact sequence $1 \to N \to G \to H \to 1$ is split (i.e., if $G$ is a semi-direct product $N \rtimes H$), then we say that the embedding problem is \textit{split} and abbreviate it by $(N\rtimes H, R)$.

A \textit{proper solution} of a differential embedding problem as above is a Picard-Vessiot ring $S/F$ with differential Galois group $G$ and an embedding of differential rings $R \subseteq S$ such that the following diagram commutes: 
\[\xymatrix{ G  \ar@{->}[d]^\cong \ar@{->>}[rr]  && H \ar@{->}[d]_\cong \\
            \underline{\Aut}^\del(S/F) \ar@{->>}[rr]^{\operatorname{res}}&    & \underline{\Aut}^\del(R/F)} \]

\begin{lem}\label{ebp-solution-criterion}
    Let $G\to H$ and $R/F$ determine a differential embedding problem as above and let $S/F$ be a Picard-Vessiot ring with differential Galois group $G$. Then there exists an embedding of differential rings $R\subseteq S$ constituting a proper solution to the differential embedding problem if and only if there exists an isomorphism of differential $H_F$-torsors $\Spec(R) \cong \Spec(S^{N_F})$.
\end{lem}

\begin{proof}
    If $R\subseteq S$ constitutes a proper solution,
then we have isomorphisms of differential $H_F$-torsors $\Spec(R) \cong \Spec(S)/N_F \cong \Spec(S^{N_F})$
by Proposition \ref{prop X/N diff} and Lemma~\ref{prop X/N}\ref{normal part}.

    Conversely, an isomorphism of $H_F$-torsors $\Spec(R) \cong \Spec(S^{N_F})$ gives rise to a commutative diagram
        \begin{equation} \label{eqn: dia for aut} \xymatrix{ \underline{\Aut}^\del(S^{N_F}/F)
     \ar^{\cong}@{->}[rr] &    &      \underline{\Aut}^\del(R/F) \\
    & H \ar_{\cong}[ru] \ar^{\cong}[lu] &}
    \end{equation}
    As the inner diagrams in
    $$
    \xymatrix{
        & \underline{\Aut}^\del(R/F) \\
        \underline{\Aut}^\del(S/F) \ar[ru] \ar[r] & \underline{\Aut}^\del(S^{N_F}/F) \ar_\cong[u] \\
        G\ar@{->>}[r] \ar^\cong[u] & H \ar_\cong[u]
    }    
    $$
    commute, also the outer diagram commutes. It follows from (\ref{eqn: dia for aut}) that we have solved the embedding problem.
\end{proof}   

Differential embedding problems not only provide information about
which differential Galois groups arise over a given differential field,
but also encode information about how the Picard-Vessiot extensions of
that field fit together.
As in ordinary Galois theory, the assertion that all split embedding problems over some field $F$ have proper solutions implies that all groups occur as (differential) Galois groups, since one can take $H$ to be trivial.  
In Theorem~\ref{thm ebp}
below we show that proper solutions to split differential embedding problems can be obtained from solutions to patching problems.

\begin{lem}\label{lem ebp1}
 Let $(F,F_1,F_2,F_0)$ be a differential diamond such that $C_{F_0}=C_F$. Let $R/F$ be a Picard-Vessiot ring such that $R\subseteq F_1$. Then the compositum $F_2R\subseteq F_0$ is a Picard-Vessiot ring over $F_2$ with the same differential Galois group as $R/F$.  Moreover, $F_2R$ is isomorphic to $F_2 \otimes_F R$ under the natural map $F_2 \otimes_F R \to F_2R$.
\end{lem}

\begin{proof}
Let $G$ be the differential Galois group of $R/F$. By \cite[Lemma 1.7]{BHH}, the compositum $F_2R$ is a Picard-Vessiot ring over $F_2$ and its differential Galois group $H$ is a subgroup of $G$. 
By the differential Galois correspondence (see \cite[Theorem~4.4]{Dyc} and the discussion at the beginning of 
Section~\ref{subsec PV}), in order to prove $H=G$
it suffices to show that $\Frac(R)^H=F$.
This equality follows from the containments $\Frac(R)^H\subseteq \Frac(R) \subseteq F_1$ and
$\Frac(R)^H \subseteq \Frac(F_2R)^H=F_2$, since $F_1 \cap F_2=F$.  So the first assertion of the lemma holds.  The final assertion follows from the fact that the natural surjection $F_2 \otimes_F R \to F_2R$ corresponds to a homomorphism of $G_{F_2}$-torsors, which must therefore be an isomorphism.
\end{proof}

\smallskip

Note that instead of citing the Galois correspondence from \cite[Theorem~4.4]{Dyc} in the above proof, one could use
the observation after Proposition~\ref{prop X/N diff}.  Namely,
$X:=\Spec(R)$ is a $G$-torsor over $F$ 
and $X_2:=\Spec(F_2R)$ is an $H$-torsor over $F_2$, where the actions of $H \subseteq G$ are compatible.  By Proposition~\ref{prop X/N}\ref{X/H part} we may consider the quotient schemes $X/H$ and $X_2/H=\Spec(F_2)$.
These are quasi-projective varieties, and are integral since $X$ and $X_2$ are the spectra of integral domains.
We have an inclusion of function fields $F(X/H) \subseteq F(X) = \Frac(R) \subset F_1$,
as well as an inclusion $F(X/H) \subseteq F_2(X_2/H) = F_2$.  Thus $F \subseteq F(X/H) \subseteq F_1 \cap F_2 = F$; i.e., $F(X/H) = F$.  By the second part of the observation after Proposition~\ref{prop X/N diff}), $H=G$.

\smallskip

To avoid burdening the notation, we sometimes drop the base change subscripts on groups in the remainder of this section if there is no possibility of confusion, especially when the group appears in a subscript or superscript.  For example, for field extensions $L/F$, we write expressions such as 
$\Ind_H^G(Y_L)$ for $\Ind_{H_L}^{G_L}(Y_L)$. We consider the following situation:

\begin{hyp}\label{hypothesis}
Let $(F,F_1,F_2,F_0)$ be a differential diamond with the factorization property and write $K=C_F$.  Let $G$ be a linear algebraic group over $K$ of the form $N \rtimes H$.  Let $S_1 = \Ind_N^G(R_1)$ for some Picard-Vessiot ring $R_1/F_1$ with differential Galois group $N_{C_{F_1}}$ such that $R_1 \subseteq F_0$.
Let 
$R/F$ be a Picard-Vessiot ring with differential Galois group $H$ such that $R \subseteq F_1 \subseteq F_0$; write $R_2 = F_2 \otimes_F R$; and let
$S_2 = \Ind_H^G(R_2) = F_2 \otimes_F \Ind_H^G(R)$. 
For $i=1,2$ let $\Theta_i: F_0 \otimes_{F_i} S_i \to F_0[G]$ be the induced differential isomorphism defined in Lemma \ref{lem theta}, using $R_i \subseteq F_0$.
\end{hyp}

\begin{prop}\label{prop S^N} 
In the situation of Hypothesis~\ref{hypothesis}, let $(S,\Phi_1,\Phi_2)$ be the solution of the patching problem $(S_1,S_2,F_0[G], \Theta_1, \Theta_2)$ of differential $G$-torsors over $(F,F_1,F_2,F_0)$ (see Theorem~\ref{thm diff solution}). Then the ring of $N_F$-invariants $S^{N_F}$ is isomorphic to $R$ as a differential $H_F$-torsor.  
\end{prop}

\begin{proof}
By taking $N$-invariants, the given patching problem gives rise to a patching problem $(S_1^N,S_2^N,F_0[H], \bar\Theta_1, \bar\Theta_2)$ of differential $H$-torsors over $(F,F_1,F_2,F_0)$, and the given solution gives rise to a solution $(S^N,\bar\Phi_1,\bar\Phi_2)$ to that patching problem.
We also have a patching problem of differential $H$-torsors given by 
$(F_1 \otimes_F R,F_2 \otimes_F R,F_0 \otimes_F R, \Omega_1, \Omega_2)$, where 
$\Omega_i$ is the natural isomorphism $F_0 \otimes_{F_i} (F_i \otimes_F R) \to F_0 \otimes_F R$.
A solution to this latter patching problem is $(R,\Psi_1,\Psi_2)$, where $\Psi_i$ is the identity on $F_i \otimes_F R$, for $i=1,2$. 
But solutions to patching problems are unique up to isomorphism (given by inverse limit, and the restriction of the derivations; see Theorem~\ref{thm diff solution}).  So it suffices to show that 
the above two patching problems of differential $H$-torsors are isomorphic.
That is, for $j=0,1,2$ we want to find differential isomorphisms $\Lambda_j$ between the respective rings in the two $H$-patching problems that carry $\Omega_i$ to $\bar\Theta_i$ for $i=1,2$.

By Corollary~\ref{ind cor}\ref{ind N mod N}, $S_1^N = (\Ind_N^G(R_1))^N = F_1[H]$, where we identify
$F_1[H]$ with $F_1 \otimes_{F_1} F_1[H] \subseteq R_1 \otimes_{F_1} F_1[G]$.  The differential isomorphism $\bar\Theta_1$ is given by $F_0 \otimes_{F_1} S_1^N = F_0 \otimes_{F_1} F_1[H] \to F_0[H]$.
Meanwhile, since $R_2 = F_2 \otimes_F R$, we have
$S_2 = F_2\otimes_F  \Ind_H^G(R) = 
(R_2 \otimes_F F[G])^{H_{F_2}}= \Ind_H^G(R_2)$, and so
$S_2^N = (\Ind_H^G(R_2))^N = \Ind_H^H(R_2)$ by Corollary~\ref{ind cor}\ref{ind H mod N}.  
The differential isomorphism $\bar\Theta_2$ is given by
$F_0 \otimes_{F_2} S_2^N = F_0 \otimes_{F_2} \Ind_H^H(R_2) = F_0 \otimes_F \Ind_H^H R  \,\stackrel{\til\Theta_0}{\to}\, F_0[H]$, where $\til\Theta_0$ is the differential isomorphism given by Lemma \ref{lem theta} using that $R \subseteq F_0$.  Also, since $R \subseteq F_1$, Lemma \ref{lem theta} yields a differential isomorphism $\til\Theta_1:F_1 \otimes_F \Ind_H^H(R) \to F_1[H] = S_1^N$; and this map is the restriction of $\til\Theta_0$. 
Let $\til\Theta_2$ be the identity map 
$F_2 \otimes_F \Ind_H^H(R) \to \Ind_H^H R_2 = S_2^N$.
We then have a commutative diagram
\[
\xymatrix  { 
F_0 \otimes_{F_1} S_1^N \ar[r]^{\bar\Theta_1} & F_0[H] & F_0 \otimes_{F_2} S_2^N 
\ar[l]_{\bar\Theta_2} \\
F_0 \otimes_{F_1}(F_1 \otimes_F \Ind_H^H(R)) \ar[r]^-{\sim} \ar[u]_{\id_{F_0} \otimes \til\Theta_1}
& F_0 \otimes_F \Ind_H^H(R) \ar[u]_{\til\Theta_0} & 
F_0 \otimes_{F_2} (F_2 \otimes_F \Ind_H^H(R)) \ar[l]_-{\sim} \ar[u]_{\id_{F_0} \otimes \til\Theta_2} \\
F_0\otimes_{F_1}(F_1 \otimes_F R) \ar[u]_{\id_{F_0} \otimes \id_{F_1} \otimes \rho_H} \ar[r]^-{\Omega_1} & F_0 \otimes_F R 
 \ar[u]_{\id_{F_0} \otimes \rho_H} 
& F_0\otimes_{F_2}(F_2 \otimes_F R) \ar[u]_{\id_{F_0} \otimes \id_{F_2} \otimes \rho_H} \ar[l]_-{\Omega_2} 
}
\] 
of differential isomorphisms, where $\rho_H:R \to \Ind_H^H(R)$
is the co-action map associated to $H$ (see Remark~\ref{rem ind GG}).
Setting $\Lambda_i = \til\Theta_i \circ (\id \otimes \rho)$ for $i=0,1,2$
then yields the assertion.
\end{proof}

\begin{thm}\label{thm ebp}
In the situation of Hypothesis \ref{hypothesis}, let $(S,\Phi_1,\Phi_2)$ be a solution of the patching problem $(S_1,S_2,F_0[G], \Theta_1, \Theta_2)$ of differential $G$-torsors. Then $S/F$ is a Picard-Vessiot ring with differential Galois group $G$, and hence is a proper solution to the split differential embedding problem given by $G=N \rtimes H$ and the Picard-Vessiot ring $R/F$.  
\end{thm}

\begin{proof}
Once the first assertion is shown, the second assertion follows from Proposition~\ref{prop S^N} together with Lemma~\ref{ebp-solution-criterion}.  By Corollary~\ref{cor PV equiv}\ref{PV equiv a}, to prove the first assertion it suffices to show that $S$ has no new constants and is differentially simple.

We first prove that $C_S=K$. Recall that $S_1=\Ind_{N_{F_1}}^{G_{F_1}}(R_1) = (R_1\otimes_{F_1} F_1[G])^{N_{F_1}}=(R_1\otimes_{C_{F_1}} C_{F_1}[G])^{N_{F_1}}$. By the assumptions, $R_1/F_1$ is a Picard-Vessiot ring, hence $C_{R_1}=C_{F_1}$ and thus the ring of constants of $R_1\otimes_{C_{F_1}} C_{F_1}[G]$ equals $C_{F_1}\otimes_{C_{F_1}} C_{F_1}[G]$. So $C_{S_1}=(C_{F_1}\otimes_{C_{F_1}} C_{F_1}[G])^{N_{C_{F_1}}}$, and the explicit description of $\Theta_1$ in Lemma \ref{lem theta} implies $\Theta_1(C_{S_1})=C_{F_1}[G]^{N_{C_{F_1}}}$. Hence $C_{\Theta_1(S_1)}=\Theta_1(C_{S_1})=C_{F_1}[G]^{N_{C_{F_1}}}\subseteq F_0[G]^{N_{F_0}}$. Here the invariants are taken with respect to the action given in Lemma~\ref{lem left right}(\ref{other action}).  But by Lemma~\ref{lem left right}\ref{lem same stability}, since $N$ is normal these are the same as the $N$-invariants under the action given in Lemma~\ref{lem left right}(\ref{torsor action}) (which agrees with the torsor action on $S$).  We thus have 
$C_S \subseteq S \cap C_{\Theta_1(S_1)} \subseteq  S\cap F_0[G]^{N_{F_0}}=S^{N_{F}}$ and hence $C_S=C_{S^{N_{F}}}$. By Proposition \ref{prop S^N}, there is a differential isomorphism $S^{N_{F}}\cong R$, hence $C_{S^{N_{F}}}=C_R$. By the assumptions, $R/F$ is a Picard-Vessiot ring, so $C_R=K$ and we conclude $C_S=K$.

\smallskip

Next, we show that $S$ is a simple differential ring.  Let $I$ be a maximal differential ideal of $S$.  By Lemma~\ref{lem patching} (with $H_1=N$), the ideal $I_0 := F_0 \otimes_F I \subseteq F_0 \otimes_F S$ is $N$-stable with respect to the action considered there; or equivalently with respect to the torsor action, by Lemma~\ref{lem left right}\ref{lem same stability}.  Since $F_0 \otimes_F S$ is faithfully flat over $S$, the ideal $I$ is the contraction of its extension $I_0$ to $ F_0 \otimes_F S$, and it is therefore $N$-stable.    
Moreover, $I \cap S^{N_F}$ is a proper differential ideal in $S^{N_F}$ and so $I\cap S^N=(0)$, since $S^N\cong R$ is differentially simple. Therefore, Lemma~\ref{lem intersection invariants} implies $I=(0)$ and thus $S$ is differentially simple. 
\end{proof}

\begin{rem} \label{rem patching thm implication}
In the situation that $C_{F_0}=C_F$, Theorem~\ref{thm ebp} follows from Theorem~\ref{thm patching PVR} (in the case of a differential diamond) via 
Lemma~\ref{lem ebp1}, since that lemma implies that $R_2/F_2$ is a Picard-Vessiot ring with differential Galois group $H$.
\end{rem}

We conclude this section by summarizing the content of Proposition \ref{prop S^N} and Theorem \ref{thm ebp} in the following

\begin{thm}\label{thm ebp new}
Let $(F,F_1,F_2,F_0)$ be a differential diamond with the factorization property and let $(N\rtimes H, R)$ be a split differential embedding problem over $F$ with the property that $R\subseteq F_1$. Assume further that there exists a Picard-Vessiot ring $R_1/F_1$ with differential Galois group $N_{C_{F_1}}$ and with $R_1\subseteq F_0$. Then there exists a proper solution to the differential embedding problem $(N\rtimes H, R)$ over $F$.	
\end{thm}
\begin{proof} Let $S_1 = \Ind_{N_{F_1}}^{G_{F_1}}(R_1)$,
	let $R_2=F_2 \otimes_F R$, let $S_2 = \Ind_{H_{F_2}}^{G_{F_2}}(R_2) = 
	F_2 \otimes_F \Ind_{H_F}^{G_F}(R)$, and let $\Theta_i: F_0 \otimes_{F_i} S_i \to F_0[G]$ be the isomorphisms as explained in Hypothesis \ref{hypothesis}. Then $(S_1,S_2,F_0[G],\Theta_1,\Theta_2)$ defines a patching problem of 
	differential $G_F$-torsors; and this has a solution $(Z,\Phi_1,\Phi_2)$ 
	that is unique up to isomorphism, by Theorem \ref{thm diff solution}.  Write 
	$Z=\Spec(S)$.  Then Theorem \ref{thm ebp} asserts that $S$ is a 
	Picard-Vessiot ring over $F$ with differential Galois group $G$, and it 
	is a proper solution to the given split differential embedding problem.
\end{proof}

\section{Differential embedding problems over complex function fields}\label{sec: Complex embedding problems}

Results about differential embedding problems, as considered in Section~\ref{subsec EP}, were obtained in \cite{Oberlies}
for the case $F=C(x)$ with $\del=d/dx$ and $C$ algebraically closed, by building on results of Kovacic (\cite{Kovacic}, \cite{Kovacic1}).
In \cite{Oberlies}, it was shown that there are proper solutions to certain types of differential embedding problems, including all those whose kernel is connected; but the general case has remained open.  

As an application of Theorem~\ref{thm ebp new}, we show in this section that the assumption that $H$ is connected can be dropped for $F=\C(x)$, and that more generally every differential embedding problem can be solved over the field of functions $F$ of any compact Riemann surface (i.e., complex curve).  Here we can take any non-trivial $\C$-linear derivation on $F$, or equivalently any derivation on $F$ for which the constants are $\C$.

The next result, based on an idea that is often called the ``Kovacic trick'', is a more precise version of an assertion in \cite{BHH}.

\begin{prop}\label{Kovacic trick}
Let $F$ be a differential field of characteristic zero and write $K=C_F$. Let $L/F$ be a differential field extension that is finitely generated over $F$ with $C_L=K$.  Let $L'$ be the algebraic closure of $F$ in $L$ and let $L''$ be the normal closure of $L'$ in $\bar{F}$. Set $d=[L'':F]$ and $m=\trd(L/F)+1$.  Let $G$ be a linear algebraic group defined over $K$ and let $R/F$ be a Picard-Vessiot ring with differential Galois group $G^{2m+2d}$.  Then there is a subring $R_0$ of $R$ such that $R_0/F$ is a Picard-Vessiot ring with differential Galois group $G$, and such that $R_0 \otimes_F L$ is a Picard-Vessiot ring over $L$ with differential Galois group $G$. 
\end{prop}

\begin{proof}
This assertion was shown in the proof of Theorem~4.12 of \cite{BHH}, though the statement of that result asserted a bit less.  Namely, in the first step of that proof of that theorem, it was shown that with notation as above, $R$ contains a subring $R'$ which is a Picard-Vessiot ring having differential Galois group $G^{2m}$, with the additional properties that $F' \otimes_F L'$ is a field, where $F'$ is the algebraic closure of $F$ in $E:=\Frac(R')$; and that $K$ is algebraically closed in $F' \otimes_F L'$.  In the second step of that proof, it was shown that $E \otimes_F L$ is an integral domain and that $K$ is algebraically closed in the fraction field $\til E$ of that domain.  In the third step, it was shown that $R'$ contains a Picard-Vessiot ring $R_0/F$ (called $R_i$ there) with differential Galois group $G$, such that the fraction field $\til E_0$ of $\til R_0 := R_0 \otimes_F L \subseteq \til E$ is the compositum $E_0L \subseteq \til E$, and such that $C_{\til E_0} = K$.  Finally, in the conclusion of the proof, it was observed that $\til R_0$ is a Picard-Vessiot ring over $L$ with differential Galois group $G$.
\end{proof}

We also recall the following from \cite{BHH} (see the statement and proof of Lemma~4.2 there):

\begin{lem} \label{change of derivation}
Let $(F,\del)$ be a differential field, let $a \in F^\times$,  let $\del'=a\del$, and let $K$ be the constants of $(F,\del)$ (or equivalently, of $(F,\del')$).  Let $G$ be a linear algebraic group over $K$, 
let $A \in F^{n \times n}$, and let $R$ be a Picard-Vessiot ring over $(F,\del)$ for the differential equation $\del y=Ay$ with differential Galois group $G$.  Then $R$ is a Picard-Vessiot ring over $(F,\del')$ for the differential equation $\del' y=aAy$ with differential Galois group $G$.
\end{lem}

In order to prove our main result, we consider fields of meromorphic functions on open subsets of Riemann surfaces, especially the Riemann sphere $\P^1_{\C} = \C \cup \{\infty\}$, and we will apply our patching and embedding problem results using these fields.  Consider the open disc $U \subset \C$ of radius $c>0$ about the origin, and let $F_1$ be the field of meromorphic functions on $U$.  Thus $\C(x) \subset F_1 \subset \C((x))$, and the standard derivation $d/dx$ on $\C((x))$ restricts to the complex derivative $d/dx$ on $F_1$ and to the standard derivation $d/dx$ on $\C(x)$.

\begin{lem}\label{lem ebp2} 
Let $c>0$, and let $\del$ be the derivation $d/dx$ on the field $F_1$ of meromorphic functions on the open disc $U = \{ x \in \C \mid |x| < c\}$.  Let $F$ be a differential subfield of $F_1$ with $C_F={\mathbb C}$, let $A \in F^{n \times n}$, and let $R/F$ be a Picard-Vessiot ring for the differential equation $\del(y)=Ay$.  If all the entries of $A$ are holomorphic on $U$, then $R$ embeds as a differential subring of $F_1$.
\end{lem}

\begin{proof}
By \cite[Theorem~11.2]{Forster},
there exists a fundamental solution matrix $Z$ for $A$ with entries in $F_1$ (in fact holomorphic on $U$). Since $C_{F_1}={\mathbb C}$, which equals the field of constants of $F$, the ring 
$F[Z,\det(Z)^{-1}]\subseteq F_1$ is a Picard-Vessiot ring over $F$ for the differential equation $\del(y)=Ay$ 
(notation as in Section~\ref{subsec PV}). The assertion then follows from the uniqueness of Picard-Vessiot rings over algebraically closed fields of constants. 
\end{proof}

\begin{lem} \label{int-fact lem}
Let $F$ be a one-variable function field over $\C$, or equivalently the field of meromorphic functions on a compact Riemann surface $\X$.  
Let $O_1,O_2$ be connected metric open subsets of $\X$ such that $O_i \ne \X$, $O_1 \cup O_2 = \X$, and $O_0 := O_1 \cap O_2$ is connected.  Let $F_i$ be the field of meromorphic functions on 
$O_i$.  Then 
\renewcommand{\theenumi}{(\alph{enumi})}
\renewcommand{\labelenumi}{(\alph{enumi})}
\begin{enumerate}
\item \label{field intersection}
As subfields of $F_0$, $F_1\cap F_2 = F$, the field of meromorphic functions on $\X$.
\item \label{field factorization}
For every $n \ge 1$, every element of $\GL_n(F_0)$ can be written as $A_2^{-1}A_1$ with 
$A_i \in \GL_n(F_i)$.
\end{enumerate}
\end{lem}

\begin{proof}
Since being meromorphic is a local property, and since $O_1 \cup O_2=\X$, the first assertion 
follows from the fact that every meromorphic function on $\X$ is a rational function on $\X$.
For the second assertion, let $\H$ be the sheaf of holomorphic functions on $\X$ in the 
complex metric topology.  Then $F_i$ is the fraction field of $R_i := \H(O_i)$, the ring of holomorphic 
functions on $O_i$.  Let $A \in \GL_n(F_0)$.

\medskip

{\em Case 1:} $A \in \GL_n(R_0)$.  

Let $M_i$ be a free $R_i$-module of rank $n$, for $i=1,2$, say with bases $B_1,B_2$. 
Thus $B_i$ is also an $F_i$-basis of the vector space $M_i \otimes_{R_i} F_i$.
Consider the locally free 
$\H$-module $\M$ on $\X$ with $\M(O_i)=M_i$ for $i=1,2$, and with transition matrix $A \in \GL_n(R_0)$ 
on $O_0$ between $B_1,B_2$ (i.e., $B_1=B_2A$).  Since $\H$ is coherent, so is $\M$, being locally free of finite rank.  By 
\cite[D\'efinition~2, Proposition~10, Th\'eor\`eme~3]{GAGA}, there is an equivalence of categories 
$\F \mapsto \F^{\mathrm h}$ from the  
coherent $\O$-modules on $\X$ to coherent $\H$-modules on $\X$, 
satisfying $\F^{\mathrm h}(U) = \H(U) \otimes_{\O(U)} \F(U)$ for every Zariski open subset
$U \subseteq \X$.  (Here $\O$ is the sheaf of regular functions on $\X$ in the Zariski topology.)  Thus $\M = \F^{\mathrm h}$ for some coherent sheaf $\F$ of $\O$-modules on $\X$.  Moreover, since $\M$ is 
locally free, so is $\F$ (see \cite{GAGA}, bottom of page~31).     

For $i=1,2$ choose a point $P_i \in \X \smallsetminus O_i$, and let $U_i = \X \smallsetminus \{P_i\}$.  
Also let $U_0 = U_1 \cap U_2 = \X \smallsetminus \{P_1,P_2\}$.  Since $\F$ is locally free, 
there is a non-empty Zariski open subset $U \subset U_0$ such that $\F(U)$ is free of rank $n$ over $\O(U)$, say with basis $B$.  Each element of $B$ has only finitely many poles on $U_i$, viewing
$F \otimes_{\O(U)} \F(U) = F \otimes_{\O(U_i)} \F(U_i)$.
By Riemann-Roch, for $i=1,2$ there exists a non-zero element 
$f_i \in \O(U_i) \subset F$ such that the elements of
$f_iB$ lie in $\F(U_i)$.  Here $f_iB$ is an $F$-basis of $F \otimes_{\O(U_i)} \F(U_i)$.
Since $O_i \subset U_i$, the set $f_iB$ is also an $F_i$-basis of $F_i \otimes_{\H(O_i)} \M(O_i)$.
Let $C_i \in \GL_n(F_i)$ be the transition matrix between the bases $B_i$ and $f_iB$ of $F_i \otimes_{\H(O_i)} \M(O_i)$; i.e., $B_i=(f_iB)C_i$.  Thus 
$B_1 = f_1BC_1 = f_1(f_2^{-1}B_2C_2^{-1})C_1 \in \GL_n(F_0)$.  Taking $A_1 = f_1f_2^{-1}C_1 \in \GL_n(F_1)$ and $A_2=C_2 \in \GL_n(F_2)$ 
yields $A_2^{-1}A_1 = B_2^{-1}B_1=A$, 
completing the proof of Case~1.

\medskip

{\em Case 2:} General case.

The entries of $A$ are meromorphic functions on $O_0$, as is $\det(A)$.  So the set $\Sigma \subset O_0$ consisting of the poles of the entries of $A$ and the zeroes of $\det(A)$ is a discrete subset of $O_0$.  
Thus the limit points of $\Sigma$ in $\X$ lie in $\partial O_1 \cup \partial O_2$, where
$\partial O_i = \bar O_i \smallsetminus O_i$ is the boundary of $O_i$.  
Since $O_i$ is open and $O_1 \cup O_2 = \X$, it follows that
$\partial O_1$ and $\partial O_2$ are disjoint closed subsets of $\X$, with 
$\partial O_1 \subset O_2$ and $\partial O_2 \subset O_1$.
Using the collar neighborhood theorem, we see that there are 
metric open neighborhoods $N_i \subset \X$ of $\partial O_i$
(for $i=1,2$) such that the following properties hold: 
\begin{itemize}
\item $\bar N_1$ is disjoint from $\bar N_2$; 
\item the open sets  
$\til O_i := O_i \smallsetminus (O_i \cap \bar N_i)$ are connected for $i=1,2$; 
\item the three intersections 
$\til O_0 := \til O_1 \cap \til O_2$, $\til O_1 \cap O_2$, and $O_1 \cap \til O_2$ are each connected; 
\item $\til O_1 \cup \til O_2 = \X$; and 
\item $\til \Sigma := \Sigma \cap \til O_0$ is finite.
\end{itemize}  

Let $O_i' := \til O_i \smallsetminus \til \Sigma$ for $i=0,1$, and let $O_2' := \til O_2$.  Then $O_i'$ is a connected open set for $i=0,1,2$; $O_1' \cap O_2' = O_0'$; 
$O_1' \cup O_2' = \X$; and $\Sigma$ is disjoint from $O_0'$.  Thus $A \in \GL_n(R_0')$, where $R_0' = \H(O_0')$. Let $O^- = O_1' \cap O_2$ and let $O^+ = O_1 \cap O_2'$; these are connected open subsets of $\X$.
Write $F_i',F^\pm$ for the field of meromorphic functions on $O_i', O^\pm$, respectively.  By Case~1, there exist $A_i \in \GL_n(F_i')$, for $i=1,2$, such that $A = A_2^{-1}A_1$ in $\GL_n(F_0')$.  It remains to show that $A_i \in \GL_n(F_i)$.  Now $A \in \GL_n(F_0) \subset \GL_n(F^+)$; $A_1 \in \GL_n(F_1')$; and $A_2 \in \GL_n(F_2') \subset \GL_n(F^+)$.  Since $A_1 = A_2A \in \GL_n(F^+)$, and since $F_1' \cap F^+ = F_1$,
it follows that $A_1 \in \GL_n(F_1' \cap F^+) = \GL_n(F_1)$.  Similarly, $F_2' \cap F^- = F_2$;
so $A_2 = A_1A^{-1}$ lies in $\GL_n(F_2' \cap F^-) = \GL_n(F_2)$.
\end{proof}

\begin{prop}\label{prop sdep application}
Let $F$ be a one-variable function field over $\C$, and let $\del$ be a derivation on $F$ with constant field $\C$.  Then every split differential embedding problem over $(F,\del)$ has a proper solution. 
\end{prop}

\begin{proof}
A given split differential embedding problem consists of a semi-direct product $G=N\rtimes H$ of linear algebraic groups and a Picard-Vessiot ring $R/F$ with differential Galois group $H$ for some differential equation $\del(y)=Ay$ over $F$.   

Consider the smooth complex projective curve $\X$ with function field $F$; we may also view this as a compact Riemann surface.  By taking a non-constant rational function on $\X$, we obtain a finite morphism $\phi:\X \to \P^1_\C$, corresponding to a branched cover of Riemann surfaces, and also corresponding to an inclusion $\iota:\C(x) \hookrightarrow F$ of function fields.  If we write $\del'$ for the derivation on $F$ induced via $\iota$ from the derivation $d/dx$ on $\C(x)$, then $\del = g(x)\del'$, where $g(x) := \del(\iota(x)) \in F$ is non-zero (and where we use that the space of derivations of $F$ over $K$ is one dimensional).  So by Lemma~\ref{change of derivation}, 
in order to prove the result we may assume that $\del$ is the derivation on $F$ that extends the derivation $d/dx$ on $\C(x)$.  

Away from a finite subset of $\X$, $\phi$ is unramified and the entries of $A$ are holomorphic.  After a translation $x \mapsto x+c$, we may assume that the fiber over the point $(x=0)$ on $\P^1_\C$ does not contain any point in that finite set.  So there is an open disc $D$ around $0 \in \P^1_\C$ 
such that $\phi^{-1}(D)$ also does not meet that finite set; and then $\phi^{-1}(D) \subset \X$, being unramified over $D$, consists of finitely many disjoint copies of $D$.  Call one of those copies $\hat O$, and let $P$ be the unique point in $\hat O \cap \phi^{-1}(0)$.  The map $\phi$ then defines a differential isomorphism between the fields of meromorphic functions on $\hat O$ and on $D$, which we identify and call $\hat F$.  
Applying Lemma~\ref{lem ebp2} to the inclusion of differential fields $F \subset \hat F$, we obtain an inclusion $R \subseteq \hat F$.

Choose a non-constant regular function on the complex affine curve $\X \smallsetminus \{P\}$; this defines a finite morphism $\pi:\X \to \P^1_\C$ such that $P$ is the unique point mapping to $\infty \in \P^1_\C$.  So there is an open disc $D' \subset \P^1_\C$ centered about $\infty$, whose closure contains no branch point of $\pi$ other than $\infty$, such that $O_1 := \pi^{-1}(D')$ is homeomorphic to an open disc, and such that the closure of $O_1$ is contained in $\hat O$.   
Since $O_1$ is contained in $\hat O$, the field of meromorphic functions $F_1$ on $O_1$ contains $\hat F$; and thus 
$R \subseteq F_1$.  Also, 
the map $\pi$ defines an inclusion $j:\C(x) \hookrightarrow F$ (not the same as the above inclusion $\iota$ defined by $\phi$).  With respect to this inclusion, we may view $F$ as a finite extension of $\C(x)$; let $d$ be the degree of its Galois closure over $\C(x)$.  

It is known that every linear algebraic group over $\C$ is a differential Galois group over $\C(x)$ (\cite{Tretkoff}; see also \cite[Theorem~5.12]{vdPutSinger}).  So there is a matrix $A'$ over $\C(x)$, and a Picard-Vessiot ring $\til R'$ over $(\C(x),d/dx)$ for the differential equation $\del y=A' y$, with 
differential Galois group $N^{2d}$.  Here $A'$ is holomorphic on some open disc in $\A^1_\C$.  A transformation of the form $x \mapsto ax+b$ (with $a \in \C^\times, b \in \C$) takes this open disc to an open disc $D''$ that is centered at the origin and contains $\P^1_\C \smallsetminus D'$.  Thus $D' \cup D'' = \P^1_\C$.  Note that $D' \cap D''$ is homeomorphic to an annulus, as is its inverse image $O_0$ under $\pi$.  Also, $\P^1_\C \smallsetminus D''$ is homeomorphic to a closed disc, as is its inverse image under $\pi$.  Hence 
$O_2 := \pi^{-1}(D'')$, which is the complement in $\X$ of this inverse image, is connected. 

The above transformation $x \mapsto ax+b$ carries $d/dx$ to $a\cdot d/dx$; carries $A'$ to a matrix $A''$ whose entries are holomorphic on $D''$; and carries $\til R'$ to a Picard-Vessiot ring $\til R''$ over $(\C(x),a\cdot d/dx)$ with differential Galois group $N^{2d}$.  By Lemma~\ref{change of derivation}, $\til R''$ is also a Picard-Vessiot ring 
over $(\C(x),d/dx)$ for the differential equation $dy/dx = a^{-1}A''y$, with differential Galois group $N^{2d}$.
Since the entries of $a^{-1}A''$ are also holomorphic on the open disc $D''$, it follows from Lemma~\ref{lem ebp2} that $\til R''$ is contained in the field of meromorphic functions on $D''$.  By Proposition~\ref{Kovacic trick}, 
there is a subring $\til R_0$ of $\til R''$ such that $\til R_0$ is a Picard-Vessiot ring over $(\C(x),d/dx)$ with differential Galois group $N$, and such that $R' := \til R_0 \otimes_{\C(x)} F$ is a Picard-Vessiot ring over $(F,\del)$ with differential Galois group $N$.  Since 
$\til R''$ is contained in the field $F''$ of meromorphic functions on $D''$, so is its subring $\til R_0$.
Tensoring this latter containment of $\C(x)$-algebras with $F$ (which is flat over $\C(x)$), we obtain an 
inclusion $R' \subseteq F'' \otimes_{\C(x)} F = F_2$, where $F_2$ is 
the field of meromorphic functions on $O_2 = \pi^{-1}(D'')$.  

Let $F_0$ denote the field of meromorphic functions on $O_0 := O_1 \cap O_2$. By Lemma~\ref{int-fact lem}, $(F,F_1,F_2,F_0)$ is a differential diamond with the factorization property. Since $R'\subseteq F_2$ and  $C_F=C_{F_0}=\C$, we may apply Lemma~\ref{lem ebp1} and obtain that the compositum $R_1=F_1R' \subseteq F_0$ is a Picard-Vessiot ring over $F_1$ with differential Galois group $N$. Recall that $R\subseteq F_1$. We conclude that by Theorem~\ref{thm ebp new}, the differential embedding problem given by $G=N\rtimes H$ together with $R$ has a proper solution. 
\end{proof}

\begin{prop} \label{prop: SEP to EP}
Let $F$ be a one-variable function field over an algebraically closed field
$K$ of characteristic zero, and let $\del$ be a derivation on $F$ with constant field $K$. 
Suppose that every split differential embedding problem over $F$ has a proper solution. Then the same holds for \textbf{\textit{every}} differential embedding problem over $F$.
\end{prop}

\begin{proof}
Consider a differential embedding problem given by an exact sequence of linear algebraic $K$-groups 
$1 \to N \to G \to H \to 1$ and a Picard-Vessiot ring $R$ for $H$ over $F$.  As in Proposition~\ref{prop PVR}, $Z := \Spec(R)$ is a simple differential $H$-torsor over $F$, corresponding to an element $\alpha \in H^1(F,H)$ under the functorial bijection between torsors and cohomology classes (e.g.\ see \cite[I.5.2 and III.1.1]{Serre:CG}).  Since $K$ is algebraically closed, 
it follows from \cite[II.3.3.Ex.~3 and II.3.1.Proposition~6(a)]{Serre:CG} that 
$F$ is a field of dimension one
(or equivalently, of cohomological dimension 
one, since $\cha(F)=0$ [Ser97, II.3.1]). 
Hence $F$ satisfies the hypotheses of
\cite[III.2.4, Theorem~3]{Serre:CG}.  Thus Corollary~2 of that theorem applies; i.e., the surjection 
$G \to H$ induces a surjection $H^1(F,G) \to H^1(F,H)$.  

Let $\tilde \alpha \in H^1(F,G)$ be an element that maps to $\alpha$, and let $X$ be the corresponding $G$-torsor over $F$.  We thus have a morphism of $G$-spaces $X \to Z$ that is constant on $N$-orbits (since $G$ acts on $Z$ through $H=G/N$).  By the universal mapping property of the quotient torsor (see Proposition~\ref{prop X/N} and the discussion just before), this morphism factors through $X/N$.  The resulting map $\iota:X/N \to Z$ is an 
$H$-torsor morphism, which is automatically an isomorphism.  
The isomorphism $\iota$ defines a differential structure on $X/N$.
By Proposition~\ref{prop: moving differential structures}\ref{moving b},
there is a differential structure on $X$ with respect to which 
the quotient morphism $\psi:X \to X/N$ is differential.  
Let $\pi: X \to Z$ be the composition $\iota \circ \psi$, which is then also differential.  

By Proposition~\ref{prop: torsors are ind}, there is a closed subgroup $J$ of $G$ such that
$X = \Ind_J^G Y$ for some simple differential $J$-torsor $Y$ such that the fraction field of the coordinate ring $S=F[Y]$ has no new constants.  Thus $S$ is a Picard-Vessiot ring for $J$ over $F$, by Proposition~\ref{prop PVR}\ref{PVR b}.  By the defining property of induced torsors, 
we have a $J$-equivariant inclusion $Y \hookrightarrow X$.  Applying $\pi$  gives a $J/(N \cap J)$-equivariant inclusion $Y/(N \cap J) \to X/N \,\tilde\longrightarrow\, Z$.  Here $Y/(N \cap J)$ is a torsor under $J/(N \cap J) = JN/N$, and $Z$ is a torsor under $H = G/N$, and the above injection is equivariant with respect to the natural inclusion $JN/N \subseteq G/N$.  (Here $JN$ is the closed subgroup of $G$ generated by $J$ and $N$.)  Also, $Z \cong X/N$ is of the form $(\Ind_J^G Y)/N = \Ind_{J/(N \cap J)}^{G/N} \bigl(Y/(N \cap J)\bigr) = \Ind_{JN/N}^{G/N} \bigl(Y/(N \cap J)\bigr)$, using Proposition~\ref{lem ind mod}.
But $Z = \Spec(R)$ is a simple differential $H$-torsor with no new constants, since $R$ is a Picard-Vessiot ring for $H$.  So Lemma~\ref{lem: simple not ind} implies that $JN/N=G/N$; i.e., $JN=G$ and the restriction of $G \to G/N = H$ to $J$ is surjective.  So the above $JN/N$-equivariant inclusion $Y/(N \cap J) \to Z$ is actually an inclusion of $G/N$-torsors, and it is therefore an isomorphism.  Hence the composition $Y \to Y/(N \cap J) \,\tilde\longrightarrow\, Z$ is surjective; i.e., $Z$ is a quotient of $Y$ as a differential torsor.  

Let $\tilde G$ be the semi-direct product $N \rtimes J$, where $J$ acts on $N$ via conjugation (as subgroups of $G$).  Let $M$ be the kernel of the surjection $\phi:\tilde G \to G$ 
given by $(n,j) \mapsto nj$. Let $\bar\phi:J \to H$ be the composition $J \hookrightarrow G \to G/N = H$.
We then have the following commutative diagram of groups with exact rows and columns, and where the middle row is split exact:
\[
\xymatrix  {
         &                 & 1 \ar[d]         & 1 \ar[d] & \\
         & 1 \ar[r] \ar[d] & M  \ar[r] \ar[d] & N \cap J  \ar[r] \ar[d] & 1 \\
1 \ar[r] & N \ar[r] \ar[d]^{=} & \tilde G \ar[r] \ar[d]^\phi & J \ar[r] \ar[d]^{\bar\phi} & 1 \\
1 \ar[r] & N \ar[r] \ar[d] & G \ar[r] \ar[d] & H \ar[r]  \ar[d] & 1 \\
         & 1               & 1                          & 1
}
\]
Here the map $M \to N \cap J$ is an isomorphism, given by $(a^{-1},a) \mapsto a$ for $a \in N \cap J$.

Consider the split differential embedding problem given by the middle row and the Picard-Vessiot ring $S$ for $J$.  
By the hypothesis of the proposition, this has a proper solution, given by a Picard-Vessiot ring $\tilde R$ for $\tilde G$ that contains $S$ as a differential subring.  Its spectrum is a differential $\tilde G$-torsor $\tilde W$ such that $\tilde W/N \,\tilde\longrightarrow\, Y$ as differential $J$-torsors.  Since $Y/(N \cap J) \to X/N \,\tilde\longrightarrow\, Z$, and since the above diagram commutes, the differential $G$-torsor $W := \tilde W/M$ satisfies $W/N \,\tilde\longrightarrow\, Z$.  
By the Galois correspondence (\cite[Thm.~4.4]{Dyc} or the observation after Proposition~\ref{prop X/N diff}),
the coordinate ring $\tilde R^M$ of $W$ is a Picard-Vessiot ring for $G = \tilde G/M$ over $F$, and 
$(\tilde R^M)^N \,\tilde\longrightarrow\, R$.  So $\tilde{R}^M$ is a proper solution to the given differential embedding problem.
\end{proof}

Combining Proposition~\ref{prop sdep application} with Proposition~\ref{prop: SEP to EP}, we obtain:

\begin{thm}\label{thm ebp application}
Let $F$ be any one-variable complex function field, together with a non-trivial $\C$-linear derivation $\del$.  Then
every differential embedding problem over $(F,\del)$ has a proper solution. 
\end{thm}

\begin{ex} Let $\Gamma$ be a lattice and let $F$ be the field of elliptic functions with respect to~$\Gamma$. Then every differential embedding problem over $F$ has a proper solution. 
\end{ex}
\begin{proof}
Recall that $F$ is generated over $\C$ by the Weierstra{\ss} function $\wp$ and its derivative~$\wp'$ which satisfy the equation $(\wp')^2=4\wp^3-g_2(\Gamma)\wp-g_3(\Gamma)$. Therefore, $F$ is a one-variable function field over $\C$ and it has a non-trivial derivation, so the claim follows from Theorem~\ref{thm ebp application}.
\end{proof}

%%----------------------------------------------------------------------------------------------------------------------------------------------------------------------------------------------------------
\appendix\section{Torsors, 	quotients, and induction} \label{sec: Preliminaries on torsors}
%%----------------------------------------------------------------------------------------------------------------------------------------------------------------------------------------------------------

This appendix contains basic results about torsors, group actions, and quotients, which are used in the body of the paper in a differential context.  Many of these results are known, but for lack of a good reference we include them here.
Throughout this appendix, $K$ is an arbitrary field, and $G$ is an affine group scheme of finite type over $K$.

%-----------------------------------------------------------------------------------------------------
\subsection{$G$-spaces} \label{subsec Gspaces}
%-----------------------------------------------------------------------------------------------------

An \textit{affine $G$-space over} $K$ is an affine scheme $X$ of finite type over $K$ together with a morphism $\alpha \colon X \times_K G \to X$ such that $\alpha_R\colon X(R)\times G(R)\to X(R)$ defines a right group action of $G(R)$ on $X(R)$ for every $K$-algebra $R$.
We will usually write $x.g$ for $\alpha_R(x,g)$ when $x\in X(R)$ and $g\in G(R)$.
A \textit{morphism} of $G$-spaces is a $G$-equivariant morphism of schemes.
The $G$-space structure $X$ corresponds to a {\em co-action} homomorphism $\rho \colon K[X] \to K[X] \otimes_K K[G]$ (where $K[X]$ is the affine coordinate ring of $X$, and similarly for $K[G]$). 

In addition to this co-action $\rho$, there is also a {\em functorial left group action} of $G$ on the coordinate ring $K[X]$.  That is, for 
every $K$-algebra $R$, there is a left action
\begin{equation} \label{eq: functorial action} 
G(R)\times (K[X]\otimes_K R)\to K[X]\otimes_K R, \ (h,f)\mapsto h(f)
\end{equation}
of the group $G(R)$ on $K[X]\otimes_K R$, and these are compatible.  To define the action in (\ref{eq: functorial action}), first note that we can identify $K[X]\otimes_K R$ with the morphisms from $X_R$ to $\mathbb{A}^1_R$.  Namely, an element $f\in K[X]\otimes_K R$ is interpreted as a morphism $f\colon X_R\to\mathbb{A}^1_R$ by sending an element $x\in X_R(S)=\Hom_R(K[X]\otimes_K R,S)$ to $x(f)$ for every $R$-algebra $S$. So for $g \in G(R)$ and $f$ as above, we can define $g(f)$ by
$$g(f)(x)=f(x.g)$$
for every $R$-algebra $S$ and every $x\in X_R(S)$. Alternatively, $g(f)$ can be described as the image of $f\in K[X]\otimes_K R$ under the composition
\begin{equation} \label{eq: H-action}
K[X]\otimes_K R\xrightarrow{\rho\otimes\operatorname{id_R}} K[X]\otimes_K K[G]\otimes_K R\to K[X]\otimes_K R,
\end{equation}
where the last map is given by $a\otimes b\otimes r\mapsto a\otimes g(b)r$ for $a\in K[X]$, $b\in K[G]$, $r\in R$ and $g\in G(R)=\Hom_K(K[G],R)$, and where $g(b)$ denotes the image of $b$ under $g:K[G] \to R$.

Given an affine $G$-space $X$ and a closed subgroup $H$ of $G$, we may consider $X$ as an $H$-space.  
We then obtain a functorial left action of $H$ on $K[X]$; this is the restriction of the above functorial action of $G$.  An element $f\in K[X]$ is $H$-{\em invariant} if $h(f\otimes 1)=f\otimes 1$ under the left action of $H(R)$ on $K[X]\otimes_K R$, for every $K$-algebra $R$ and $h \in H(R)$.
These $H$-invariant elements form a subring denoted by $K[X]^H$, the {\em (functorial) invariants} in $K[X]$.  

Observe that $f\in K[X]$ is invariant under $H$ if and only if $\rho_H(f)=f\otimes 1$, where
$\rho_H \colon K[X]\to K[X]\otimes_K K[H]$ is the co-action corresponding to the action of $H$ on $X$.  The co-action maps for $G$ and $H$ are related by the identity 
$\rho_H = \pi \circ \rho$, where $\pi:K[G]\to K[H]$ is the homomorphism corresponding to the inclusion of $H$ into $G$.  This identity shows that the map in the display (\ref{eq: H-action}) above agrees with the corresponding map with $G$ replaced by $H$ and $\rho$ by $\rho_H$, if the element $g \in G(R)$ that is used in (\ref{eq: H-action}) lies in $H(R)$.

\begin{lem} \label{lemma:invariant ideal}
	For an ideal $J\trianglelefteq K[X]$ the following conditions are equivalent:
\renewcommand{\theenumi}{(\alph{enumi})}
\renewcommand{\labelenumi}{(\alph{enumi})}
	\begin{enumerate}
		\item \label{stable cond a} $H(R)(J\otimes_K R)\subseteq J\otimes_K R$ for all $K$-algebras $R$.
		\item \label{stable cond b} $Z=\Spec(K[X]/J)$ is $H$-stable; i.e., $Z(R).H(R)\subseteq Z(R)$ for all $K$-algebras $R$.
		\item \label{stable cond c} $\rho_H(J)\subseteq J\otimes_K K[H]$.
	\end{enumerate}
\end{lem}

\begin{proof}
	To see \ref{stable cond a} $\Rightarrow$ \ref{stable cond b}, let $R$ be a $K$-algebra and let $z\in Z(R),\ h\in H(R)$ and $f\in J\otimes_K R$. Since
	$$Z(R)=\{x\in X(R)\,|\, f(x)=0 \ \mathrm{for\ all} \ f\in J\otimes_K R\}$$
	and $h(f)\in J\otimes_K R$,
	we have $f(z.h)=h(f)(z)=0$. Therefore $z.h\in Z(R)$.
	
	To prove \ref{stable cond b} $\Rightarrow$ \ref{stable cond a}, note that
	$$J\otimes_K R=\{f\in K[X]\otimes_K R\,|\, f(z)=0\  \text{for\ all} \ R\text{-algebras } S \ \text{and\ all} \ z\in Z(S) \}$$
	If $f\in J\otimes_K R$, $h\in H(R)$, $z\in Z(S)$, then
	$h(f)(z)=f(z.h)=0$ because $z.h\in Z(S)$. Therefore $h(f)\in J\otimes_K R$.
	
    In order to prove \ref{stable cond a} $\Rightarrow$ \ref{stable cond c}, 
    consider display (\ref{eq: H-action}) with $G$ replaced by $H$, taking $R=K[H]$,
    and letting $g$ be the element $h \in H(R)$ that corresponds to the identity map
    in $\Hom_K(K[H],R)$.  Taking $j \in J$ and using the display in that situation,
    we see that $h(j \otimes 1) = \rho_H(j)$.
    But part~\ref{stable cond a} says in particular that $h(j \otimes 1) \in J \otimes_K R$.
    So \ref{stable cond c} follows.
	
	Finally, \ref{stable cond c} $\Rightarrow$ \ref{stable cond a} follows from the description of the $H$-action in (\ref{eq: H-action}) with $G$ replaced by $H$.
\end{proof}

An ideal $J\trianglelefteq K[X]$ satisfying the equivalent conditions of Lemma \ref{lemma:invariant ideal} is called \textit{$H$-stable}.
We will frequently use that every ideal generated by elements in $K[X]^H$ is $H$-stable.

If $H_1,\dots,H_m$ are closed subgroups of $G$, the smallest closed subgroup of $G$ that contains all $H_i$'s is called the closed subgroup {\em generated} by the $H_i$'s.

\begin{lem}\label{lem invariants}
Consider $X=G$ as a $G$-space via the action given by the formula $x.g=g^{-1}x$. Let $H$ be the closed subgroup of $G$ generated by the closed subgroups $H_1,\dots,H_m$ of $G$. Then the following holds:
\renewcommand{\theenumi}{(\alph{enumi})}
\renewcommand{\labelenumi}{(\alph{enumi})}
	\begin{enumerate}
		\item \label{stable ideals a} $\bigcap \limits_{i=1}^m K[G]^{H_i}=K[G]^H$ and $K[G]^G=K$.
		\item \label{stable ideals b} Every ideal in $K[G]$ that is $H_i$-stable for all $i=1,\dots,m$ is $H$-stable.
		\item \label{stable ideals c} The only proper $G$-stable ideal in $K[G]$ is the zero ideal.
	\end{enumerate}
\end{lem}

\begin{proof}
For part \ref{stable ideals a}, observe that
$K[G]^{H}\subseteq \bigcap \limits_{i=1}^m K[G]^{H_i}$. For any $K$-algebra $R$ write
	$$H'(R):=\left\{h\in  G(R)\,|\, h(f\otimes 1)=f\otimes 1 \ \text{for\ all}  \ f\in \bigcap_{i=1}^m K[G]^{H_i} \right\}.$$
	We claim that $H'$ is a closed subgroup of $H$. To see this it suffices to show that the condition  $h(f\otimes 1)=f\otimes 1$ is a closed condition for any non-zero $f\in K[G]$: Let $(f_i)$ be a $K$-basis of $K[X]=K[G]$ such that $f_0=f$ and write $\rho(f)=\sum f_i\otimes g_i\in K[X]\otimes_K K[G]$. Then $h(f\otimes 1)=\sum f_i\otimes h(g_i)$ for $h\in G(R)=\Hom_K(K[G],R)$. Therefore,  $h(f\otimes 1)=f\otimes 1$ if and only if $h(g_0)=1$ and $h(g_i)=0$ for $i\neq 0$. Thus $H'$ is the closed subscheme of $G$ defined by $g_0-1$ and $g_i$ for $i\neq 0$, proving the claim. 
	
	Because $H'$ is a closed subgroup of $G$ containing $H_1,\ldots,H_m$ we must have $H\subseteq H'$; i.e., $\bigcap \limits_{i=1}^m K[G]^{H_i}\subseteq K[G]^{H}$.  This proves the first part of \ref{stable ideals a}.  For the second part of \ref{stable ideals a}, if
$f\in K[G]^G$, then $f(g)=f(1.g^{-1})=g^{-1}(f)(1)=f(1)\in K$ for all $g\in G(R)$. So $f\in K$.
	
	To prove \ref{stable ideals b}, 
let $J\trianglelefteq K[G]$	be an ideal, and for any $K$-algebra $R$ write
		$$Z(R):=\{h\in G(R)|\ h(J\otimes_K R)\subseteq J\otimes_K R\}.$$
We claim that $Z$ is closed in $G$, and hence so is 
$Z \cap Z^{-1} \subseteq G$. But 
		$$(Z \cap Z^{-1})(R)=\{h\in G(R)|\ h(J\otimes_K R)=J\otimes_K R\}$$
for any $K$-algebra $R$; so $Z \cap Z^{-1}$ is a 
subgroup
 of $G$.  Thus if $J$ is $H_i$-stable for all $i$ then $J$ is also $H$-stable, by definition of $H$; and \ref{stable ideals b} would then follow.

To prove the claim, let $(f_i)_{i\in \I}$ be a $K$-basis of $K[G]$ such that $\I=\I_1\cup \I_2$ is the disjoint union of $\I_1$ and $\I_2$ and $(f_i)_{i\in \I_1}$ is a basis of $J$. Let $f\in J$ and write $\rho(f)=\sum f_i\otimes g_i\in K[G]\otimes_K K[G]$. Then $h(f\otimes 1)=\sum f_i\otimes h(g_i)$ for $h\in G(R)=\Hom_K(K[G],K)$. Therefore $h(f\otimes 1)\in J\otimes_K R$ if and only if $h(g_i)=0$ for $i\in \I_2$. This proves the claim, and part~\ref{stable ideals b}.
	
	Part \ref{stable ideals c} is clear from Lemma \ref{lemma:invariant ideal}\ref{stable cond b}, since $G$ is the only $G$-stable closed subscheme of $G$. 	
\end{proof}

\begin{lem}\label{lem left right}
Let $N$ be a normal closed subgroup of $G$. We consider two right actions of $N$ on $G$: 
	\begin{enumerate}[(i)]
		\item \label{torsor action} $G\times N \to G$ given by the formula $(g,n) \mapsto g\!\cdot\!n$
		\item \label{other action} $G\times N \to G$ given by the formula $(g,n) \mapsto n^{-1}\!\cdot\!g$.
	\end{enumerate}
	Let $\rho_1 \colon K[G] \to K[G]\otimes K[N]$ and $\rho_2 \colon K[G] \to K[G]\otimes K[N]$ denote the homomorphisms corresponding to (\ref{torsor action}) and (\ref{other action}). Then the following hold.
\renewcommand{\theenumi}{(\alph{enumi})}
\renewcommand{\labelenumi}{(\alph{enumi})}
	\begin{enumerate}
		\item \label{lem same invariants} The two actions give rise to the same ring of invariants $K[G]^N \subseteq K[G]$:
		\[\{f \in K[G] \mid \rho_1(f)=f\otimes 1\} = \{f \in K[G] \mid \rho_2(f)=f\otimes 1\}. \]
		\item \label{lem same stability} Let $J \trianglelefteq K[G]$ be an ideal. Then $J$ is $N$-stable with respect to the action (\ref{torsor action}) if and only if $J$ is $N$-stable with respect to the action (\ref{other action}), that is, $\rho_1(J) \subseteq J \otimes K[N]$ if and only if $\rho_2(J) \subseteq J \otimes K[N]$. 
	\end{enumerate}
\end{lem}

\begin{proof}
	To prove \ref{lem same invariants} we have to show that for $f\in K[G]$ we have $f(gn)=f(g)$ for all $n\in N(R)$ and $g\in G(S)$ for all $K$-algebras $R$ and $R$-algebras $S$ if and only if $f(n^{-1}g)=f(g)$ for all $n\in N(R)$ and $g\in G(S)$ for all $K$-algebras $R$ and $R$-algebras $S$.
	
	First assume that $f(gn)=f(g)$. Since $g^{-1}n^{-1}g\in N(S)$ we find $f(n^{-1}g)=f(gg^{-1}n^{-1}g)=f(g)$.
	Conversely, if  $f(n^{-1}g)=f(g)$, then $f(gn)=f(gng^{-1}g)=f(g)$ because $gng^{-1}\in N(S)$.
	
	To prove \ref{lem same stability}, according to Lemma \ref{lemma:invariant ideal}, we have to show that for a closed subscheme $Z$ of $G$ we have
	$Z(R)N(R)\subseteq Z(R)$ for all $K$-algebras $R$ if and only if $N(R)Z(R)\subseteq Z(R)$ for all $K$-algebras $R$.
	
	First assume that $Z(R)N(R)\subseteq Z(R)$. For $z\in Z(R)$ and $n\in N(R)$ the element $z^{-1}nz$ belongs to $N(R)$ and therefore
	$nz=zz^{-1}nz\in Z(R)$.
	
	Conversely, assume that $N(R)Z(R)\subseteq Z(R)$, then
	$zn=znz^{-1}z\in Z(R)$ because $znz^{-1}\in N(R)$.
\end{proof}
 
%-----------------------------------------------------------------------------------------------------
\subsection{Quotients}\label{quotients}
%-----------------------------------------------------------------------------------------------------

In order to discuss the quotient of a $G$-space $X$ by the action of $G$, we recall
some basic facts about sheaves in the faithfully flat topology. Let $K$ be a field and let $\F$ be a functor from the category of $K$-algebras to the category of sets. Then $\F$ is a \textit{sheaf} if 
$\F(A \times B) = \F(A) \times \F(B)$ for every pair of $K$-algebras $A,B$, and 
if for every faithfully flat homomorphism $R \to S$ of $K$-algebras the sequence $\F(R)\to \F(S)\rightrightarrows \F(S\otimes_R S)$ is exact (i.e., an equalizer of sets). \textit{Morphisms of sheaves} are defined as morphisms of functors. Every quasi-projective scheme over $K$ (in particular, every affine $K$-scheme) is a sheaf.

If $\F$ is a functor from the category of $K$-algebras to the category of sets, then there exists a sheaf $\tilde \F$ and a morphism $\iota \colon \F \to \tilde \F$ that is universal among morphisms from $\F$ to sheaves; i.e., for every morphism $\phi \colon \F \to \F'$ from $\F$ to a sheaf $\F'$, there exists a unique morphism $\tilde \phi \colon \tilde \F \to \F'$ such that $\phi=\tilde \phi \circ \iota$. See for example \cite[III, \S 1, Theorem 1.8]{DG} for a proof. The sheaf $\tilde \F$ is called the \textit{sheafification of $\F$}. If $\phi \colon \F \to \F'$ is a morphism of functors, then the universal property of $\tilde \F$ applied to the morphism $\F \to \F' \to \tilde \F'$ yields a morphism $\tilde \phi \colon \tilde \F \to \tilde \F'$ that is called the \textit{sheafification of $\phi$}. 

Let $G$ be an affine group scheme of finite type over $K$ and let $X$ be an affine $G$-space over $K$. 
We view $X$ as a sheaf.
A morphism $\pi \colon X \to Y$ of sheaves is \textit{constant on $G$-orbits} if the following diagram commutes:
\[\xymatrix{ X\times G  \ar@{->}[dr]_{\pi\circ \operatorname{pr}_1} \ar@{->}[rr]^\alpha   && X \ar@{->}[dl]^{\pi} \\
            & Y  & } \]

 We write $X\qq G$ for the functor of $G$-orbits on $X$; i.e., $(X\qq G)(R)=X(R)/G(R)$ for every $K$-algebra $R$. Let $\operatorname{pr} \colon X\to X\qq G$ be the canonical projection. Set $Y=\widetilde{X\qq G}$ and let $\pi=\tilde{\operatorname{pr}} \colon X\to Y$ be the sheafification of $\operatorname{pr}$. 
 It is easy to see that $Y$  together with $\pi\colon X\to Y$ is the \textit{universal quotient in the category of sheaves}, that is, $\pi$ is constant on $G$-orbits and for every other morphism $\phi \colon X\to Y'$ of sheaves that is constant on $G$-orbits, there exists a unique morphism $\psi \colon Y \to Y'$ such that $\psi \circ \pi=\phi$.  In the situation where $\widetilde{X\qq G}$ is represented by a quasi-projective $F$-scheme, we
will write $X/G$ for that scheme.  A case of particular interest in where $X/G$ is an affine scheme; in that situation we have the following result:

\begin{prop}\label{quotient}
Let $G$ be an affine group scheme of finite type over $K$ and let $X$ be an affine $G$-space over $K$. If the quotient $$Y=\widetilde{X\qq G}$$ is an affine scheme, then $K[Y]\cong K[X]^G$, and the projection $\pi \colon X\to Y$ is given by the inclusion $K[Y] \cong K[X]^G \subseteq K[X]$.
\end{prop}

\begin{proof}
Let $\rho \colon K[X]\to K[X]\otimes_K K[G]$ be the homomorphism corresponding to the action $\alpha: X\times G \to X$; i.e., $\alpha^*=\rho$. 

Set $Z=\Spec(K[X]^G)$ and let $\pi_Z \colon X \to Z$ be the morphism defined by the inclusion $\pi_Z^* \colon K[X]^G\hookrightarrow K[X]$. By definition of $K[X]^G$, the composition $\rho\circ \pi_Z^*$ maps $f$ to $f\otimes 1$ for every $f \in K[X]^G$. Hence $\pi_Z \circ \alpha \colon X\times G \to Z$ equals $\pi_Z$ applied to the first factor, and $\pi_Z$ is thus constant on $G$-orbits. It is easy to check that $Z$ is the universal quotient of $X$ by $G$ in the category of affine schemes. 

By assumption, $Y$ is an affine scheme and it is the universal quotient in the category of sheaves, therefore it is also the universal quotient in the category of affine schemes. We conclude $Y \cong Z$ and thus $K[Y]\cong K[Z]=K[X]^G$. 
\end{proof}

\begin{rem}
For every affine $G$-space $X$, there exists a  universal quotient of $X$ mod $G$ in the category of affine schemes, namely $\Spec(K[X]^G)$.  However, this does not imply that this is a quotient with ``good properties'', since the category of affine schemes is in general too small for ``good quotients''. For example the multiplicative group $G=\mathbb{G}_m$ acts on the affine line $X=\mathbb{A}^1$, but $K[X]^G=K$, so the affine quotient $\Spec(K[X]^G)$ is trivial and does not contain any information on the $G$-orbits. 
\end{rem}

A non-empty affine $G$-space $X$ is called a \textit{$G$-torsor} if for every $K$-algebra $R$ and all $x,y \in X(R)$ there exists a unique $g \in G(R)$ with $y=x.g$.  Equivalently, the condition is that the morphism $(\alpha,\pr_1):X \times G  \to X \times X$ is an isomorphism, where $\pr_1$ is the projection onto the first component.
A torsor is called \textit{trivial} if it is isomorphic to the $G$-torsor $G$ (where the $G$-action on $G$ is defined by right multiplication). A torsor $X$ is trivial if and only if $X(K)\neq \varnothing$.  Over an algebraically closed field, every torsor is trivial; and so for general $K$, there is a finite extension $L/K$ such that the base change $X_L$ is trivial as a $G_L$-torsor.   

The next proposition assures that certain quotients of torsors are well-behaved.  We first recall the basics of flat descent theory for schemes, in the special case of field extensions.  Let $L/K$ be a field extension.  Then {\em descent data} for an $L$-scheme $Z$ consists of an isomorphism $\phi:Z \times_K L \to L \times_K Z$ such that $\phi_2=\phi_1 \circ \phi_3$, where $\phi_1 = \id_L \times \phi:L \times_K Z \times_K L \to L \times_K L \times_K Z$ and similarly for $\phi_2:Z \times_K L \times_K L \to L \times_K L \times_K Z$ and $\phi_3: Z \times_K L \times_K L\to L  \times_K Z\times_K L$ (where $\id_L$ is inserted into the second or third factor respectively).  A $K$-scheme $Y$ induces an $L$-scheme $Z := Y \times_K L$ together with descent data; and conversely an $L$-scheme $Z$ together with descent data is induced by a $K$-scheme $Y$ that is unique up to isomorphism.  See \cite[VIII]{SGA1} and \cite[I.2]{milneEtCoh} for more details.  Similarly, a sheaf over $K$ induces a sheaf over $L$ together with descent data in the analogous sense.
Concerning uniqueness, the sheaf axiom implies that a sheaf $\F$ over $K$ is determined by its base change to 
$L$ together with the induced descent data, since for any $K$-algebra $A$, the set
$\F(A)$ is the equalizer of $\F(A \otimes_K L) \rightrightarrows \F(A \otimes_K L \otimes_K L)$ and similarly for morphisms.  See also \cite[III, Section~3.4, Proposition~6.2]{DG}.

\begin{prop}\label{prop X/N}
Let $X$ be a $G$-torsor for an affine group scheme $G$ of finite type over $K$.
\renewcommand{\theenumi}{(\alph{enumi})}
\renewcommand{\labelenumi}{(\alph{enumi})}
\begin{enumerate}
\item \label{X/H part}
If $H$ is a closed subgroup of $G$, then $X/H$ exists; i.e., $\widetilde{X\qq H}$ is represented by a quasi-projective scheme.
\item \label{normal part}
Let $N\trianglelefteq G$ be a closed normal subgroup. Then 
$X/N$ is isomorphic to the affine scheme $\Spec(K[X]^N)$, and it is a $(G/N)$-torsor.  In addition, the homomorphism $\rho_X \colon K[X]\to K[X] \otimes K[G]$ corresponding to the $G$-action on $X$ restricts to a homomorphism $\rho_Y \colon K[X]^N \to K[X]^N \otimes K[G]^N$ and this homomorphism corresponds to the $(G/N)$-action on $Y$. Furthermore, $K[X]$ is faithfully flat over $K[X]^N$. 
\end{enumerate}
\end{prop}

\begin{proof}
For part~\ref{X/H part}, the sheaf $\widetilde{G\qq H}$ is represented by a $K$-scheme (see \cite[III, Section~3.5, Th\'eor\`eme~5.4]{DG}), and this scheme is quasi-projective by \cite[Corollaire VI.2.6]{Ray}.  That is, the quotient $G/H$ exists.
The $G$-torsor $X$ becomes trivial over some finite field extension $L/K$, and this induces an isomorphism of sheaves between $(\widetilde{X\qq H})_L$ and $(G/H)_L$.  The extension of constants from $K$ to $L$ induces descent data for the sheaf
$(\widetilde{X\qq H})_L$.  Since 
$(\widetilde{X\qq H})_L$ is isomorphic to a quasi-projective $L$-scheme (viz.\ $(G/H)_L$), it follows from faithfully flat descent for schemes (\cite[VIII, Corollaire~7.7]{SGA1}) that there is a quasi-projective $K$-scheme $Y$ that induces this $L$-scheme together with the above descent data.  
But $Y$ and 
$\widetilde{X\qq H}$ are then both sheaves that induce $Y_L$ together with the above descent data.  
So by the uniqueness assertion before the statement of the proposition,
it follows that $Y$ and 
$\widetilde{X\qq H}$ are isomorphic as sheaves.  That is, $Y = X/H$.  This proves part~\ref{X/H part}. 

Next we turn to part~\ref{normal part}. 
Since $N$ is normal, the quotient $G/N$ is an affine $K$-group scheme of finite type (see \cite[III, Section~3.5, Th\'eor\`eme~5.6]{DG}); hence 
so is $(G/N)_L$.  Thus in the argument above (with $H=N$), the descended scheme $Y = X/N$ is also affine, by \cite[VIII, Th\'eor\`eme~2.1]{SGA1} (or \cite[I.2, Theorem~2.23]{milneEtCoh}).

For every $K$-algebra $R$,
the image of an element $x \in X(R)$ in $X(R)/N(R)$ will be denoted by $[x]$, and the image of $g \in G(R)$ in $G(R)/N(R)$ will be denoted by $\bar g$.  Thus 
\[ X(R)/N(R) \times G(R)/N(R) \to X(R)/N(R), \ ([x],\bar g) \mapsto [x].\bar g:=[x.g] \] 
is a well-defined group action for every $K$-algebra $R$ and it defines a morphism of functors $(X\qq N) \times (G\qq N)  \to (X\qq N)$. The sheafification $X/N \times (G/N) \to X/N$ of this morphism defines a group action of $G/N$ on $X/N$. For every $K$-algebra $R$, and for all $[x], [y] \in (X\qq N)(R)$, there exists a unique $\bar g \in (G\qq N)(R)$ with $[x]=[y].\bar g$. Therefore, 
\[ (X\qq N)(R) \times (G\qq N)(R) \to(X\qq N)(R) \times (X\qq N)(R), \ ([x],\bar g) \mapsto ([x], [x].\bar g) \] is bijective for every $K$-algebra $R$. Hence the sheafification $X/N \times (G/N) \to X/N \times X/N$ is an isomorphism, and we conclude that $X/N$ is a $(G/N)$-torsor. 

Proposition \ref{quotient} implies that $K[X/N]=K[X]^N$ (i.e., $X/N \cong \Spec(K[X]^N)$) and that the quotient map $\pi\colon X\to X/N$ is induced from the inclusion $K[X]^N \subseteq K[X]$. Similarly, the quotient map $\pi_G \colon G \to G/N$ is induced from the inclusion $K[G/N]=K[G]^N \subseteq K[G]$. By construction of the $(G/N)$-action on $X/N$, the following diagram commutes:
\[
\xymatrix{
X  \times G \ar[rr] \ar[d]^{\pi\times \pi_G} &&  X \ar[d]^{\pi}\\
 (X/N) \times (G/N) \ar[rr] &&X/N  \\ } 
\]
and we conclude that the homomorphism $\rho_X \colon K[X]\to K[X] \otimes K[G]$ corresponding to the $G$-action on $X$ restricts to a homomorphism $\rho_{X/N} \colon K[X]^N \to K[X]^N \otimes K[G]^N$ corresponding to the $(G/N)$-action on $X/N$.

Finally, we show that $K[X]$ is faithfully flat over $K[X]^N$. Fix a finite field extension $L/K$ such that $X(L)$ is non-empty. Let $x \in X(L)$ and define $y=\pi(x) \in (X/N)(L)$. Then $g \mapsto x.g$ defines an isomorphism $G_L\to X_L$, and $\bar g \mapsto y.\bar g = \pi(x.g)$ defines an isomorphism $(G/N)_L \to (X/N)_L$, in each case between trivial torsors. The isomorphism $(L\otimes_K K[X]) \to (L\otimes_K K[G])$ corresponding to $G_L \to X_L$ restricts to the isomorphism $(L\otimes_K K[X]^N)\to (L\otimes_K K[G]^N)$ corresponding to $(G/N)_L \to (X/N)_L$.  By \cite[III.3, Proposition 2.5]{DG}, $K[G]$ is faithfully flat over $K[G]^N$.  So
$L\otimes_K K[G]$ is faithfully flat over $L\otimes_K K[G]^N$, and therefore $L\otimes_K K[X]$ is faithfully flat over $L\otimes_K K[X]^N$. Using that $L/K$ is faithfully flat, we conclude that $K[X]/K[X]^N$ is faithfully flat. 
\end{proof}

The next result is used in Section~\ref{sec: Patch Embed}.

\begin{lem} \label{lem intersection invariants}
	Let $G$ be a smooth affine group scheme of finite type over a field $K$ and let $N$ be a normal closed subgroup. Let $X$ be a $G$-torsor, and let $I \trianglelefteq K[X]$ be an $N$-stable, non-zero ideal. Then $I \cap K[X]^N \neq 0$. 
\end{lem} 

\begin{proof}
	First note that by flatness of $\bar K$ over $K$, we may assume that $K$ is algebraically closed. Hence $X$ is a trivial $G$-torsor and we may assume $X=G$. Let $I \trianglelefteq K[G]$ be an $N$-stable, non-zero ideal. We need to show that $I \cap K[G]^N \neq 0$. Let $Z \subsetneq G$ be the closed subset defined by $I$. 
By Proposition~\ref{quotient},	
the quotient morphism $\pi \colon G \to G/N$ corresponds to the inclusion $K[G]^N \subseteq K[G]$ on the level of coordinate rings. Furthermore, $\pi$ is open (see Theorem 5.5.5 in \cite{Springer} and its proof), so $\pi(G\smallsetminus Z)$ is an open and non-empty subset of $G/N$. Also note that $Z$ is a union of $N$-orbits, since $I$ is $N$-stable. However, the fibers of $\pi$ are the $N$-orbits in $G$ (\cite[Cor. 5.5.4]{Springer}) and so $\pi(Z)\cap \pi(G\smallsetminus Z)=\varnothing$. Therefore, $\pi(Z) \subseteq (G/N)\smallsetminus \pi(G\smallsetminus Z)$ is not dense in $G/N$ and hence $\sqrt{I} \cap K[G]^N\neq 0$.  Choose $0 \ne f \in \sqrt{I} \cap K[G]^N$.  Then some power $f^n$ lies in $I \cap K[G]^N$.  
But $G$ is reduced, so $f^n \ne 0$ and thus $I \cap K[G]^N \neq 0$. 
\end{proof}

%-----------------------------------------------------------------------------------------------------
\subsection{Induced torsors}\label{inducedtorsors}
%-----------------------------------------------------------------------------------------------------

Let $H$ be a closed subgroup of $G$ and let $Y$ be an $H$-torsor. There is a natural way to construct a $G$-torsor $$X=\ind$$ from $Y$ (\cite[III, \S 4, 3.2]{DG}) which we recall in the proof of Proposition \ref{prop: universal property of ind} below. This follows directly from the identification of isomorphism classes of torsors with equivalence classes of $1$-cocycles, and the fact that the inclusion map $H \hookrightarrow G$ induces a map $H^1(K,H) \to H^1(K,G)$.   
Alternatively, the induced torsor $\ind$ can be characterized by a universal property:

\begin{prop} \label{prop: universal property of ind}
	Let $H$ be a closed subgroup of $G$ and let $Y$ be an $H$-torsor. Then there exists a $G$-torsor $\ind$ together with an $H$-equivariant morphism $Y\to\ind$ such that for every $H$-equivariant morphism $Y\to Z$ from $Y$ to some affine $G$-space $Z$, there exists a unique $G$-equivariant morphism $\ind\to Z$ making the diagram
	\begin{equation} \label{eq: diagram for ind}
	\xymatrix{ Y  \ar@{->}[dr] \ar@{->}[rr]   & & \ind \ar@{..>}[dl] \\
		& Z  & }
	\end{equation}
	commutative.  This $G$-torsor is unique up to a unique isomorphism, and it is given by the quotient
	$(Y \times G)/H$, where $Y \times G$ is an $H$-space under the action given by 
	$(y,g).h=(y.h, h^{-1}g)$ for every $K$-algebra $R$, and all $y \in Y(R)$, $g \in G(R)$, and $h \in H(R)$.	
\end{prop}

\begin{proof}
Set $\F=(Y\times G)\qq H$, and write $[y,g] \in \F(R)$ for the $H(R)$-orbit of an element $(y,g)\in Y(R)\times G(R)$, where $R$ is a $K$-algebra. We define a group action of $G$ on the functor $\F$ by 
$$\F(R)\times G(R)\to \F(R), \ ([y,g],\tilde g)\mapsto [y,g\tilde g] \ \text{for} \ y \in Y(R)\ \text{and}\ g,\!\tilde g \in G(R). $$
Letting $X$ be the sheafification $((Y \times G)\qq H)\widetilde{\ \ }$, 
we obtain a group action of $G$ on $X$ given by the 
sheafification $X\times G \to X$ of the above morphism. 
For every $K$-algebra $R$ and every two elements $[y,g], [y',g'] \in \F(R)$, there exists a unique $\tilde g \in G(R)$ with $[y,g\tilde g]=[y',g']$, since $Y$ is an $H$-torsor. Therefore, the map $\F(R)\times G(R)\to \F(R)\times \F(R)$ given by $([y,g],\tilde g)\mapsto ([y,g],[y,g\tilde g])$ is bijective for every $K$-algebra $R$. Thus $\F \times G \to \F \times \F$ is an isomorphism of functors, and so its sheafification 
$X \times G \to X \times X$ is an isomorphism (here we used the canonical isomorphisms 
$\tilde \F \times \tilde \F \cong \widetilde{\F \times \F}$ 
and $\tilde \F \times G \cong \tilde \F \times \tilde G \cong \widetilde{\F \times G}$). 
 
Since $Y$ is non-empty, $X$ is non-empty and we may fix a field extension $L/K$ such that $X,Y$ have an $L$-point.  Pick such a point $x\in X(L)$. Then the formula $g \mapsto x.g$ defines an isomorphism $G_L\cong X_L$. In particular, the base change $X_L$ of $X$ from $K$ to $L$ is an affine scheme. As $L$ is faithfully flat over $K$ and $X$ is a sheaf, it follows that $X$ is an affine scheme, via faithfully flat descent, as in the proof of Proposition~\ref{prop X/N}\ref{X/H part}.
It also follows that $X$ is of finite type over $K$.
So we may write $X = (Y \times G)/H$.   
Proposition \ref{quotient} implies $K[X]=(K[Y]\otimes_K K[G])^H$.  (Recall that $(K[Y]\otimes_K K[G])^H=\{ f \in K[Y]\otimes_K K[G] \mid \rho(f)=f\otimes 1\}$, where $\rho: K[Y]\otimes_K K[G] \to K[Y]\otimes_K K[G]\otimes_K K[H]$ is the co-action corresponding to the action of $H$ on $Y\times G$.)  The isomorphism $X \times G \to X \times X$ shows that $X$ is a $G$-torsor.
 	
	We define a morphism $Y\to\F$ by sending $y\in Y(R)$ to the equivalence class $[y,1]$ of $(y,1)\in (Y\times G)(R)$ for any $K$-algebra $R$. Composing with $\F\to \ind$ yields a morphism $\psi\colon Y\to\ind$. Since $[y.h,1]=[y,h]$ for $h\in H(R)$, we see that $\psi$ is $H$-equivariant. Now let $\psi'\colon Y\to Z$ be any $H$-equivariant morphism from $Y$ to an affine $G$-space $Z$. The $G$-equivariant
	morphism  $Y\times G\to Z,\ (y,g)\mapsto\psi'(y).g$ is constant on $H$-orbits and therefore induces a unique $G$-equivariant morphism
	$\ind\to Z$ making (\ref{eq: diagram for ind}) commutative, 
	by the universal property of the quotient. 
	This proves the universal property of induced torsors that was stated in the proposition.
	
Finally, the universal property satisfied by $\ind$ shows that this is unique up to a unique isomorphism.	
\end{proof}

\begin{rem} \label{rem: uniqueness of ind}
\renewcommand{\theenumi}{(\alph{enumi})}
\renewcommand{\labelenumi}{(\alph{enumi})}
\begin{enumerate}
\item \label{ind rem a}
Let $H$ be a closed subgroup of $G$. If $Y\to X$ is an $H$-equivariant morphism from an $H$-torsor $Y$ to a $G$-torsor $X$ then $X=\ind$. This is clear because $G$-equivariant morphisms between $G$-torsors are isomorphisms.
\item \label{ind rem b}
Consider the special case that $G$ is a semi-direct product $N \rtimes H$ of closed subgroups with $N$ normal;
i.e., $G(A)=N(A)\rtimes H(A)$ for every $K$-algebra $A$.
Here we may take 
$Z=Y$ and $Y \to Z$ the identity map in Proposition~\ref{prop: universal property of ind}, where $G$ acts on $Y$ through $H=G/N$.  The proposition then yields
a canonical $G$-equivariant morphism $\ind \to Y$, for which the given morphism $Y \to \ind$ is a section.
\item \label{ind rem c}
Using the universal property, it is immediate that the $G$-torsor induced from the trivial $H$-torsor is trivial.
\item \label{ind rem d}
As above, let $H$ be a closed subgroup of $G$ defined over $K$, and let $Y$ be an $H$-torsor over $K$.
If $L$ is a field extension of $K$, the universal property of induced torsors implies that 
$\bigl(\Ind_H^G(Y)\bigr)_L$ is canonically isomorphic to $\Ind_{H_L}^{G_L}(Y_L)$.
\end{enumerate}
\end{rem}

\begin{lem} \label{lem: canonical map is embedding}
	The morphism $\psi\colon Y\to X=\ind$ given by Proposition \ref{prop: universal property of ind} is a closed embedding; i.e., the dual map $\psi^*\colon K[X]\to K[Y]$ is surjective.
\end{lem}

\begin{proof}
	By flatness, we may assume that $K$ is algebraically closed. Let $y\in Y(K)$ and define an isomorphism $H\to Y$ by $h\mapsto y.h$. Similarly, let $x=\psi(y)$ and define an isomorphism $G\to X$ by $g\mapsto x.g$. Then
	$$
	\xymatrix{
		Y \ar[r] & X \\
		H \ar[r] \ar^\cong[u] & G \ar_\cong[u]
	}	
	$$
	commutes. Since $H\to G$ is a closed embedding, $Y\to X$ is also a closed embedding.	
\end{proof}

Since $Y\to\ind$ is a closed embedding, we will in the sequel identify $Y$ with a closed subscheme of $\ind$.

Let $G$ be a finite (abstract) group. Recall (see e.g., \cite[Section 18.B]{KMRT}) that a finite separable (commutative) $K$-algebra $L$ equipped with a $G$-action from the left is called a \emph{$G$-Galois algebra} if $\dim_KL=|G|$ and the invariant subfield $L^G$ equals $K$. As noted there, these algebras are precisely the coordinate rings of $G$-torsors over $K$, where we view $G$ as a constant group scheme over $K$ (i.e., with trivial action of the absolute Galois group of $K$). 
The following example shows that the notion of induced Galois algebra can be seen as a special case of inducing torsors.

\begin{ex}\label{ex induced algebra}
Let $G$ be a finite group with a subgroup $H$. For an $H$-Galois algebra $L$, there exists an \textit{induced $G$-Galois algebra} $\Ind^G_H(L)$ together with an $H$-equivariant morphism $\pi:\Ind^G_H(L)\to L$; and moreover the pair $(\Ind^G_H(L),\pi)$ is unique up to isomorphism (see \cite[Prop.~18.17]{KMRT} for details).  
Consider $G$ and $H$ as constant group schemes over $K$ and let $Y=\Spec(L)$ denote the $H$-torsor corresponding to $L$. We claim that
$$K\!\!\left[\Ind^G_H(Y)\right]\cong\Ind^G_H(L)\quad \text{as $G$-Galois algebras}.$$
To see this, note that if $H\to G$ is a morphism of algebraic groups, $Y$ an $H$-torsor, $X$ a $G$-torsor and $\psi:Y\to X$ an $H$-equivariant morphism, then $\psi^*\otimes_K R: F[Y]\otimes_K R\to F[X]\otimes_K R$ is $H(R)$-equivariant for any $K$-algebra $R$.
So, as $Y\to \Ind^G_H(Y)$ is $H$-equivariant (choosing $R=K$), so is the dual map $K\!\!\left[\Ind^G_H(Y)\right]\to L$. Therefore $K\!\!\left[\Ind^G_H(Y)\right]\cong\Ind^G_H(L)$.
\end{ex}

Motivated by this example, if more generally $Y = \Spec(R)$ in the situation of Proposition~\ref{prop: universal property of ind}, we will write $\Ind_H^G(R)$ for the coordinate ring $K[\ind]$.  Thus $\ind = \Spec(\Ind_H^G(R))$.
By Proposition~\ref{prop: universal property of ind}, $\Ind_H^G(R)$ is the invariant ring $(R \otimes_K K[G])^H$,
where the action of $H$ on $R \otimes_K K[G]$ corresponds to the geometric action of $H$ on $Y \times G$ given in that proposition.

\begin{rem} \label{rem ind GG}
If $Y = \Spec(R)$ is a $G$-torsor over $K$, then by setting $H=G$ in Proposition~\ref{prop: universal property of ind}
we see that $Y$ is canonically isomorphic to $\Ind_G^G(Y)$.  The isomorphism can be made explicit on the level of coordinate rings.  Namely, the co-action $\rho:R \to R \otimes_K K[G]$ has image $\Ind_G^G(R)=(R \otimes_K K[G])^G \subseteq R \otimes_K K[G]$, and $\rho$ defines an isomorphism of $R$ with $\Ind_G^G(R)$.  
\end{rem}

\begin{prop} \label{lem ind mod}
Let $G$ be a smooth affine group scheme of finite type over a field $K$, let $E,N$ be closed subgroups of $G$ with $N$ normal, and let $X=\Spec(R)$ be an $E$-torsor over $K$.  
Consider the functorial $G$-action on $\Ind_E^G(R)$ corresponding to the $G$-torsor structure of its spectrum.  With respect to the restriction of this action to $N$, 
$(\Ind_E^G(R))^N = \Ind_{E/(E\cap N)}^{G/N}(R^{E \cap N})$ as subalgebras of $\Ind_E^G(R)$.  
\end{prop}

\begin{proof}
By Proposition~\ref{prop: universal property of ind}, $\Ind_E^G(R) = (R \otimes_K K[G])^E$, where the $E$-invariants are taken with respect to the group action
given geometrically by $(x,g).h = (x.h,h^{-1}g)$.  By Proposition~\ref{prop X/N}, $X/(E\cap N)=\operatorname{Spec}(R^{E\cap N})$, and this is a torsor for $E/(E\cap N)$. Moreover, 
$\Ind_{E/(E\cap N)}^{G/N}(R^{E \cap N}) = (R^{E \cap N} \otimes_K K[G/N])^{E/(E\cap N)}$, with respect to the corresponding group action of $E/(E\cap N)$.  
Observe that the rings $(\Ind_E^G(R))^N$ and $\Ind_{E/(E\cap N)}^{G/N}(R^{E \cap N})$ in the statement of the proposition are indeed both contained in
$\Ind_E^G(R)$, whose spectrum is a $G$-torsor.  The restriction to $N$ of the corresponding functorial $G$-action on $\Ind_E^G(R)$ extends to $R \otimes_K K[G]$, and is given geometrically by right multiplication on the second factor.  Note that this action of $N$ commutes with the above action of $E$.  Hence
$(\Ind_E^G(R))^N$ is the ring of $E$-invariants of the $K$-algebra $R \otimes_K K[G]^N$. Equivalently, it is the ring of $E$-invariants of $(R \otimes_K K[G]^N)^{E \cap N}$, where $E \cap N$ acts via the restriction of the above action of $E$.
But $(R \otimes_K K[G]^N)^{E \cap N}  
 = R^{E \cap N} \otimes_K K[G]^N$, since the right-invariant subring $K[G]^N$ is also left-invariant under $N$ by Lemma~\ref{lem left right}(a), by normality of $N$.
The actions of $E$ on the spectra of $R^{E \cap N}$ and on $K[G]^N$
are functorially defined over each $K$-algebra, and their restrictions to
$E \cap N$ are trivial.  Hence 
the action of $E$ factors through $E/(E \cap N)$, and the above $E$-invariant ring is the same as the ring of $E/(E \cap N)$-invariants of $R^{E \cap N} \otimes_K K[G/N]$, viz.\ $\Ind_{E/(E\cap N)}^{G/N}(R^{E \cap N})$.
\end{proof}

\begin{cor} \label{ind cor}
Consider a smooth affine group scheme $G$ of finite type over a field $K$, with $G$ a semi-direct product $N \rtimes H$ of a closed normal subgroup $N$ and a closed subgroup $H$.  
\renewcommand{\theenumi}{(\alph{enumi})}
\renewcommand{\labelenumi}{(\alph{enumi})}
\begin{enumerate}
\item \label{ind N mod N}
If $\Spec(R)$ is an $N$-torsor over $K$, then $(\Ind_N^G(R))^N$ equals $K[H]$; i.e., it is the subalgebra 
$\Ind_1^{G/N}(K) = K \otimes_K K[H]$ of $\Ind_N^G(R) \subseteq R \otimes_K K[G]$.
\item \label{ind H mod N}
If $\Spec(R)$ is an $H$-torsor over $K$, then $(\Ind_H^G(R))^N$ equals $R$; i.e., it is the subalgebra 
$Ind_H^H(R) \subseteq R \otimes_K K[H]$ of $\Ind_H^G(R) \subseteq R \otimes_K K[G]$ 
(where $R$ maps isomorphically to $Ind_H^H(R)$ by the co-action map $\rho_H$ corresponding to the action of $H$; cf.\ Remark~\ref{rem ind GG}).
\end{enumerate}
\end{cor}

\begin{proof}
These are special cases of Proposition~\ref{lem ind mod}, with $E$ equal to $N$ or $H$, respectively.
\end{proof}

\medskip

\noindent Author information:

\medskip

\noindent Annette Bachmayr (n\'{e}e Maier): Mathematisches Institut der Universit\"at Bonn,
D-53115 Bonn, Germany.\\ email: {\tt bachmayr@math.uni-bonn.de}

\medskip

\noindent David Harbater: Department of Mathematics, University of Pennsylvania, Philadelphia, PA 19104-6395, USA.\\ email: {\tt harbater@math.upenn.edu}

\medskip

\noindent Julia Hartmann:  Department of Mathematics, University of Pennsylvania, Philadelphia, PA 19104-6395, USA.\\ email: {\tt hartmann@math.upenn.edu}

\medskip

\noindent Michael Wibmer: Department of Mathematics, University of Pennsylvania, Philadelphia, PA 19104-6395, USA.\\ email: {\tt wibmer@math.upenn.edu}

\end{document}